\documentclass[12pt,leqno]{amsart}
\usepackage{etoolbox}
\newbool{anonymous}
\relax

\usepackage{hyperref}

\usepackage{graphicx, verbatim}

\usepackage{amsfonts}
\usepackage{amssymb}
\usepackage{amstext}
\usepackage{amsmath}

\usepackage{verbatim, scalerel}
\usepackage{xspace}
\usepackage{graphicx}

\newcommand{\R}{\ensuremath{\mathbb R}}

\newcommand{\mean}[1]{\ensuremath{\operatorname{mean}\inp{#1}}}

\newcommand{\ra}{\ensuremath{\rightarrow}}

\newcommand{\de}{\ensuremath{\delta}}

\newcommand{\Z}{\mathbb{Z}} \newcommand{\Q}{\mathbb{Q}}  
 \newcommand{\x}{\mathbf{x}} \newcommand{\y}{\mathbf{y}}    

\newcommand{\n}{\mathbf{n}}    
    \newcommand{\one}{\mathbf{1}}

\newcommand{\la}{\lambda}

\newcommand\eps{{\epsilon}}

\newcommand{\lvol}[1]{\ensuremath{\mathrm{vol}\inp{#1}}}

\newcommand{\tg}{\hat{g}}

\newcommand{\dist}{\mathbf{d}}

\renewcommand{\i}{k}

\usepackage[english]{babel}\usepackage[utf8]{inputenc}
\usepackage[noadjust]{cite}
\usepackage{etoolbox}
\usepackage{amsmath, amsthm}
\usepackage{amsfonts, amssymb}
\usepackage{bm}
\usepackage{ifthen}

\usepackage{letltxmacro}

\usepackage{nicefrac}
\usepackage{xstring}
\usepackage{pgf}
\usepackage{tikz}
\usetikzlibrary{arrows,automata}
\usepackage[T1]{fontenc}
\usepackage[tracking,nopatch=eqnum]{microtype}
\usepackage[ttscale=0.81]{libertine}
\usepackage[libertine]{newtxmath}

\renewcommand{\mathbf}[1]{\bm{#1}}

\usepackage{geometry}
\usepackage{fullpage}

\usepackage{graphicx}

\usepackage{framed}
\usepackage{marginnote}

\usepackage{todo}

\usepackage{booktabs,siunitx}

\usepackage{hyperref}
\hypersetup{
  colorlinks=true,
  linkcolor=green!30!black,
  linktoc=all,
  citecolor=green!50!black,
  bookmarksnumbered=true
}
\newcommand{\arxiv}[1]{\href{https://arxiv.org/abs/#1}{arXiv:#1}}
\usepackage{amsthm}
\numberwithin{equation}{section}
\usepackage{cleveref}

\newtheorem{theorem}{Theorem}[section]
\newtheorem{observation}[theorem]{Observation}
\newtheorem{claim}[theorem]{Claim}

\newtheorem{lemma}[theorem]{Lemma}
\newtheorem{corollary}[theorem]{Corollary}
\newtheorem{proposition}[theorem]{Proposition}

\theoremstyle{definition}
\newtheorem{definition}[theorem]{Definition}
\newtheorem{remark}[theorem]{Remark}

\AddToHook{env/theorem/begin}{\crefalias{theorem}{theorem}}
\AddToHook{env/observation/begin}{\crefalias{theorem}{observation}}
\AddToHook{env/claim/begin}{\crefalias{theorem}{claim}}
\AddToHook{env/condition/begin}{\crefalias{theorem}{condition}}
\AddToHook{env/example/begin}{\crefalias{theorem}{example}}
\AddToHook{env/fact/begin}{\crefalias{theorem}{fact}}
\AddToHook{env/lemma/begin}{\crefalias{theorem}{lemma}}
\AddToHook{env/corollary/begin}{\crefalias{theorem}{corollary}}
\AddToHook{env/proposition/begin}{\crefalias{theorem}{proposition}}
\AddToHook{env/question/begin}{\crefalias{theorem}{question}}
\AddToHook{env/definition/begin}{\crefalias{theorem}{definition}}
\AddToHook{env/remark/begin}{\crefalias{theorem}{remark}}

\crefname{theorem}{Theorem}{Theorems}
\crefname{observation}{Observation}{Observations}
\crefname{claim}{Claim}{Claims}
\crefname{condition}{Condition}{Conditions}
\crefname{example}{Example}{Examples}
\crefname{fact}{Fact}{Facts}
\crefname{lemma}{Lemma}{Lemmas}
\crefname{corollary}{Corollary}{Corollaries}
\crefname{definition}{Definition}{Definitions}
\crefname{remark}{Remark}{Remarks}
\crefname{proposition}{Proposition}{Propositions}
\crefname{section}{Section}{Sections}
\crefname{appendix}{Appendix}{Appendices}

\newcommand{\st}[0]{\ensuremath{\;\mathbf{:}\;}}

\newcommand{\abs}[1]{\ensuremath{\left|#1\right|}}
\newcommand{\norm}[2][]{\ensuremath{\Vert #2 \Vert_{#1}}}

\newcommand{\ceil}[1]{\ensuremath{\left\lceil#1\right\rceil}}

\newcommand{\inb}[1]{\left\{#1\right\}}
\newcommand{\inp}[1]{\left(#1\right)}
\newcommand{\insq}[1]{\left[#1\right]}

\makeatletter
\newcommand*{\defeq}{\mathrel{\rlap{\raisebox{0.3ex}{$\m@th\cdot$}}\raisebox{-0.3ex}{$\m@th\cdot$}}=}
\newcommand*{\eqdef}{=
  \mathrel{\rlap{\raisebox{0.3ex}{$\m@th\cdot$}}\raisebox{-0.3ex}{$\m@th\cdot$}}}
\makeatother

\newcommand{\poly}[1]{\ensuremath{\operatorname{poly}\inp{#1}}}

\renewcommand{\vec}[1]{\mathbf{#1}}

\newcommand{\pl}{Prékopa–Leindler}

\newcommand{\ind}[1]{\ensuremath{\mathbb{I}_{#1}}}

\usepackage{mdframed} \newmdenv[linecolor=red, frametitle=Caution!, backgroundcolor=yellow!40!white, outerlinewidth=2pt, roundcorner=10pt, ]{infobox}

\newcommand{\kostka}[2]{\ensuremath{K_{#1,#2}}}
\newcommand{\opart}[2]{\ensuremath{P_{#1,#2}}}
\renewcommand{\H}{\ensuremath{\mathcal{H}}}
\newcommand{\GT}{\ensuremath{\mathrm{GT}}}
\newcommand{\tpGT}{\ensuremath{{\mathrm{pGT}}}}
\newcommand{\pGT}{\ensuremath{\widetilde{\mathrm{pGT}}}}
\newcommand{\tp}{\ensuremath{\tilde{p}}}
\newcommand{\SH}{\ensuremath{\mathrm{SH}}}
\newcommand{\pSH}{\ensuremath{\mathrm{pSH}}}
\NewDocumentCommand{\gap}{O{\lambda}}{\ensuremath{{#1}_{\mathrm{gap}}}}
\newcommand{\bdrad}{\ensuremath{\varepsilon}}

\newcommand{\rep}[1]{\ensuremath{\mathsf{rep}\inp{#1}}}

\newcommand{\tV}{\ensuremath{\tilde{V}}}

\newcommand{\grantacknowledgement}{This work was supported by the Department of
  Atomic Energy, Government of India [project number RTI4001]; by the
  Infosys-Chandrasekharan virtual center for Random Geometry at the Tata
  Institute of Fundamental Research; by a Swarna Jayanti fellowship to HN; and
  by the Science and Engineering Research Board [grant number MATRICS
  MTR/2023/001547].  The contents of this paper do not necessarily reflect the
  views of the funding agencies listed above.}

\begin{document}

\title{Deterministically approximating the volume of a Kostka polytope}
\author{Hariharan Narayanan}
\address{School of Technology and Computer Science, Tata Institute of Fundamental Research, Mumbai, Maharashtra, India}
\email{hariharan.narayanan@tifr.res.in}

\author{Piyush Srivastava}
\address{School of Technology and Computer Science, Tata Institute of Fundamental Research, Mumbai, Maharashtra, India}
\email{piyush.srivastava@tifr.res.in}

\date{}
\pagenumbering{gobble}
\maketitle
\begin{abstract}
  Polynomial-time deterministic approximation of volumes of polytopes, up to an
  approximation factor that grows at most sub-exponentially with the dimension,
  remains an open problem.  Recent work on this question has focused on
  identifying interesting classes of polytopes for which such approximation
  algorithms can be obtained.  In this paper, we focus on one such class of
  polytopes: the \emph{Kostka polytopes}.  The volumes of Kostka polytopes
  appear naturally in questions of random matrix theory, in the context of
  evaluating the probability density that a random Hermitian matrix with fixed
  spectrum $\lambda$ has a given diagonal $\mu$ (the so-called \emph{randomized
    Schur-Horn problem}): the corresponding Kostka polytope is denoted
  $\mathrm{GT}(\lambda, \mu)$.  We give a polynomial-time deterministic algorithm for
  approximating the volume of a ($\Omega(n^2)$ dimensional) Kostka polytope
  $\mathrm{GT}(\lambda, \mu)$ to within a multiplicative factor of
  $\exp(O(n\log n))$, when $\lambda$ is an integral partition with $n$ distinct parts, with
  entries bounded above by a polynomial in $n$, and $\mu$ is an integer vector
  lying in the interior of the permutohedron (i.e., convex hull of all
  permutations) of $\lambda$.  The algorithm thus gives asymptotically correct
  estimates of the log-volume of Kostka polytopes corresponding to such
  $(\lambda, \mu)$.  Our approach is based on a partition function interpretation of a
  continuous analogue of Schur polynomials.
 \end{abstract}

\tableofcontents

\newpage
\pagenumbering{arabic}
\section{Introduction}

The problem of computing the volume of a convex body, and even more
specifically, that of \emph{convex polytopes}, is a widely-studied fundamental
algorithmic problem.  It was proved by Dyer and
Frieze~\cite{dyerComplexityComputingVolume1988} that exactly computing the
volume of a polytope specified by a system of linear inequalities (with binary
rational coefficients) is \#P-hard.  Soon thereafter, the seminal paper of Dyer,
Frieze, and Kannan~\cite{dyerRandomPolynomialtimeAlgorithm1991} gave a
fully-polynomial randomized approximation scheme (FPRAS) for the volume of
convex bodies specified via ``well-guaranteed'' membership oracles or
separation oracles (and hence, in particular, also for polytopes specified by a
system of linear inequalities).  Since then, the question of \emph{randomized}
approximation of volumes of convex bodies in general, and convex polytopes in
particular, has been explored extensively; a variety of algorithms for the
problem have been analyzed, and these analyses have lead to the development of
new ideas both in convex geometry and in the analysis of Markov chains: see,
e.g., \cite{leeManifoldJoysSampling2022} for a survey of selected results and
open problems.

The problem of volume estimation, however, has also served as an important
testing ground for the power of randomness in algorithm
design~\cite{dyerComputingVolumeConvex1991}.  It is known that in the membership
oracle model, no deterministic polynomial time algorithm can approximate the
volume of a convex body in \(\R^d\) to better than an exponential
factor~\cite{elekesGeometricInequalityComplexity1986,baranyComputingVolumeDifficult1987}:
this is an unconditional result.  More precisely, it is shown in
\cite{baranyComputingVolumeDifficult1987} that with $N = d^{\Theta(1)}$ queries
to the membership oracle, one cannot deterministically estimate the volume
better than within a factor of $(\gamma d/\ln N)^{d/2}$, where $\gamma > 0$ is an absolute
constant. Deterministic polynomial time algorithms that essentially match this
lower bound are known in the well-guaranteed membership oracle model: in
particular, it is known that the volume of a convex body in \(\R^d\), given by
such an oracle, can be deterministically approximated up to an approximation
factor of \(\exp(O(d \log d))\), in time polynomial in \(d\)~\cite[p.~127 and
Figure 4.1]{GLSbook}.

The unconditional lower bounds of
\cite{elekesGeometricInequalityComplexity1986,baranyComputingVolumeDifficult1987}
for deterministic algorithms do not apply when the convex body is a polytope
specified by a system of linear inequalities.  Yet, to the best of our
knowledge, \(\exp(O(d\log d))\) still remains the best known factor of
approximation that polynomial time deterministic algorithms can achieve for
volume approximation of general polytopes in \(\R^d\).  In particular, it is not
known how to approximate the volume of an arbitrary polytope to within
\emph{subexponential} (in the dimension) factors using a deterministic
polynomial-time algorithm.  Work on deterministic approximation of volumes of
polytopes has therefore focused on understanding whether one can improve upon
the \(\exp(O(d \log d))\) approximation factor on special classes of polytopes,
starting with the early work of
Lawrence~\cite{lawrencePolytopeVolumeComputation1991} and
Barvinok~\cite{barvinokComputingVolumeCounting1993}.  Another important theme in
this line of work (including in work as recent as that of Barvinok and
Rudelson~\cite{BR22_quick}) has been that of developing semi-explicit formulas
for the volume of the polytopes being studied, perhaps involving an optimization
problem, and then studying how well the formula can be approximated (or perhaps
even exactly evaluated) by deterministic polynomial time algorithms.  This
approach also has the potential advantage of the resulting formula being useful
for analytical purposes in a manner that having a general volume approximation
algorithm, whether deterministic or randomized, might not have been.

\subsection{Our contribution}
In this paper, we contribute to this line of work, and focus on the
deterministic approximation of the volumes of \emph{Kostka polytopes}.  Let
\(\lambda\) be a non-increasing sequence of \(n\) real numbers, and let \(\mu\) be a
sequence of \(n\) real numbers so that \(\lambda\) and \(\mu\) have the same sum.  A
\emph{Gelfand Tsetlin (GT) pattern} with top row \(\lambda\) is an interlacing
triangular array of real numbers with \(\lambda\) as the top row, and the
\emph{Gelfand-Tsetlin polytope} \(\GT(\lambda)\) is the collection of all
Gelfand-Tsetlin patterns with top row \(\lambda\) (see \cref{def-kostka-polytope} for
references and a more formal definition, and \cref{fig:gt-pattern} for an
illustration).  The \emph{Kostka polytope} \(\GT(\lambda, \mu)\) is the section of
\(\GT(\lambda)\) consisting of those Gelfand-Tsetlin patterns in which, for each
\(1 \leq k \leq n\), the sum of the bottom \(k\) rows is equal to the sum of the first
\(k\) entries of \(\mu\).

The volumes of Kostka polytopes appear naturally in questions of random matrix
theory in the context of what is called the randomized Schur-Horn
problem. More precisely, the following is known (see, e.g., \cite{CoqZub},
\cite[Section 4.6]{baryshnikov_gues_2001} or \cite[Section 2]{LMV21}): first,
given a Borel set $ E \subset \R^n$, let $\mathbb{Q}_n(E)$ be the probability that
$E$ contains the diagonal of a random matrix sampled from the unitarily
invariant measure supported on the set of Hermitian matrices with fixed spectrum
$\la$.  On the other hand, suppose that we sample a random Gelfand-Tsetlin
pattern
${(x_{ij})}_{1 \leq j\leq i\leq n}$ from the normalized Lebesgue measure on the polytope
$\GT(\la)$ of all triangular interlacing arrays with top row
${(x_{n j})}_{1 \leq j \leq n} = \lambda$. Let $ \mu \in \R^n$ be a random vector whose
$k^{th}$ coordinate equals the random variable
$(\sum_{j \leq k} x_{k\,j}) - (\sum_{j \leq k-1} x_{k-1\,j})$. Given a Borel set
$ E \subset \R^n$, let $\mathbb{P}_n(E)$ denote the probability that
$\mu \in E$.  Then, \[\mathbb{P}_n(E) = \mathbb{Q}_n(E).\]
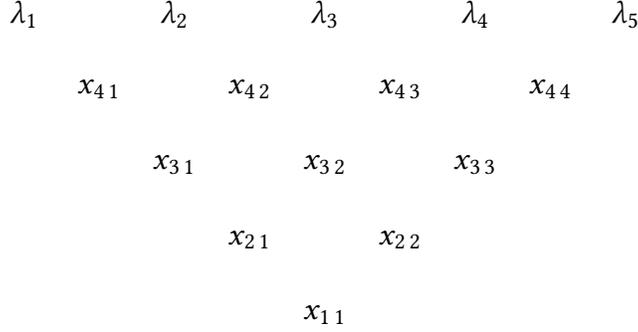
\begin{figure}[t]
  \centering
\begin{tikzpicture}
\tikzset{every node/.style={anchor=center}}

\renewcommand{\n}{5}  

\foreach \i in {1,...,\n} {
        \node (lambda\i) at (2*\i-\n-1, \n-1) {$\lambda_{\i}$};
    }

\foreach \row in {1,...,\n} {
        \foreach \col in {1,...,\row} {
            \ifnum\row<\n
\pgfmathsetmacro{\x}{2*\col-\row-1}
                \pgfmathsetmacro{\y}{\row-1}

\node at (\x, \y) {$x_{\row\,\col}$};
            \fi
        }
    }

  \end{tikzpicture}
  \caption{A Gelfand-Tsetlin pattern with top row \(\lambda\).  The top row is
    non-increasing from left to right, and each element in the other rows has a value
    that is at most the value of the element to it left in row above, and at least
    the value of the element to its right in the row above.}
  \label{fig:gt-pattern}
\end{figure}

The volume of the Gelfand-Tsetlin polytope \(\GT(\lambda)\) equals a ratio of two
Vandermonde determinants~\cite[Lemma 1.12]{baryshnikov_gues_2001}. It therefore
turns out that the probability density of \(\mathbb{P}_n\) (with respect to an
appropriate Lebesgue measure) is equal to the volume of the Kostka polytope
$\GT(\la, \mu)$ up to an exactly and efficiently computable multiplicative
constant that is independent of $\mu$.
Our goal is to \emph{deterministically} compute an approximation to the volume
of $\GT(\lambda, \mu)$ in time polynomial in $n, \abs{\lambda},$ and
$\rep{\lambda, \mu}$, where \(\abs{\lambda}\) is the sum of the entries of
\(\lambda\) and $\rep{\cdot}$ represents the input length of a collection of rational
numbers when specified as pairs of co-prime integers in binary.
(This dependence of the runtime on $\abs{\lambda}$ is analogous to that obtained
by Barvinok and Fomin~\cite[Section 4.2]{barvinok_sparse_1997} in a related
context.)

\paragraph{\textbf{Main result.}} We now describe our main result, \cref{thm:main}, in
the important special case when \(\lambda\) is a strictly decreasing \(n\)-tuple of
integers with entries bounded above by some polynomial in \(n\), and \(\mu\) is an
integer point lying in the relative interior of the permutohedron of \(\lambda\)
(i.e., the convex hull of all permutations of
\(\lambda\)).  In this case, our result states that for $S_\la(x)$ generically given by
$\frac{\det[\exp(x_i\la_j)]_{1\leq i,j \leq n}}{V(x)},$ and defined everywhere
in $\R^n$ by continuity,\footnote{Here,
  $V(x) \defeq \prod_{1 \leq i < j \leq n}(x_i - x_j)$ is a Vandermonde
  determinant.}
\begin{equation}
  \begin{aligned}[t]
    \frac{\sqrt{(n-1)!} \cdot \exp(-n)}{\lambda_1^{n-1}n^{2n}}
    &\leq \frac{\lvol{\GT(\la, \mu)}}{\inf_{x \in \R^n} S_{\lambda}(x) \exp(- x^T\mu)}\\
    & \leq {
      \sqrt{(n-1)!} \cdot \exp(n/2) \cdot \Gamma\inp{\frac{n+1}{2}}}
      \cdot \inp{\frac{16 \lambda_1 n^3}{\sqrt{\pi}}}^{n-1}.
  \end{aligned}\label{eq-main-result}
\end{equation}
Further, in \cref{sec:comp-aspects} (see \cref{thm-main-algo}), we show that in
the above setting, one can obtain a \((1 \pm 2^{-k})\)-factor multiplicative
approximation to the quantity
\(\inf_{x \in \R^n} S_{\lambda}(x) \exp(- x^T\mu)\) in time polynomial in \(n\)
and \(k\).  (Note that in the above setting, \(\abs{\lambda}\) and
\(\rep{\mu, \lambda}\) are both bounded from above by a polynomial in \(n\).)

\begin{remark}
  Note that the approximation ratio (in this case, the ratio of the upper and
  lower bounds above) is at most
  \(\exp\inp{n\cdot\inp{\frac{11}{2}\log n + 2\log \lambda_1 + O(1)}}\).  Thus, when
  \(\lambda_1\) is bounded above by a polynomial in \(n\) (as stipulated above), the
  approximation ratio is sub-exponential in the dimension of
  \(\GT(\lambda, \mu)\), which is \(\Theta(n^2)\).
\end{remark}

\begin{remark}
  As we discuss in more detail in \cref{sec:defin-prel} below, when \(\lambda\) and
  \(\mu\) have integer entries, there is a close connection between the Kostka
  polytope \(\GT(\lambda, \mu)\) and the well-studied \emph{Kostka numbers}
  $K_{\lambda, \mu}$, which count the number of semi-standard Young tableaux with shape
  \(\lambda\) and type \(\mu\), and therefore appear as coefficients in the Schur
  polynomial corresponding to \(\lambda\) (see \cref{eq-schur-poly} below). In
  particular, when \(\lambda\) and \(\mu\) have integer entries,
  \(K_{\lambda, \mu}\) is the number of integral points in the polytope
  \(\GT(\lambda, \mu)\).  This connection is what makes the case of integral
  \(\lambda\) and \(\mu\) specially important.
\end{remark}

\begin{remark}
  During the preparation of this manuscript, we were informed by Jonathan Leake
  that there is a method of approximating the Kostka numbers
  \(K_{\lambda, \mu}\) that proceeds by putting together a result of Br{\"a}nd{\'e}n,
  Leake and Pak~\cite[Theorem 5.10]{BLP} on the approximability of coefficients
  of denormalized Lorentzian polynomials, with the fact that the Schur
  polynomials belong to this class (a result of Huh, Matherne, Mészáros, and
  St.~Dizier~\cite{HMMSD22}).  For Kostka numbers, this yields a multiplicative
  $\exp(O(n\log n))$ approximation when the sizes of the parts of the partition
  are polynomial in $n$. By taking semiclassical limits with care, quantitative
  bounds on the volumes of Kostka polytopes that are similar in flavor to those
  of the present paper can be consequently obtained, but we do not pursue this
  here, and focus instead on a comparatively more direct and elementary approach
  that does not require the machinery of (denormalized) Lorentzian polynomials.
\end{remark}

\subsection{Related work}
As mentioned above, the Kostka polytope is closely related to the Kostka
numbers.  The computation of Kostka numbers with \(\lambda\) and \(\mu\)
specified in binary is known to be \#P-complete~\cite{Nar06}.  Thus, under
standard complexity theoretic assumptions, it is not possible to have a
polynomial time algorithm that computes Kostka numbers exactly. This hardness
result extends to the problem of exactly computing the volume of the Kostka
polytope, at least when \(\lambda\) and \(\mu\) are specified in binary.
Concerning the computation of Kostka numbers, it is known from the work of
Barvinok and Fomin~\cite[Corollary 4.3]{barvinok_sparse_1997} that the
collection of {\it all} the Kostka numbers corresponding to \(\lambda\) (i.e.,
as \(\mu\) varies) can be computed in time that is polynomial in the input and
output description lengths.  Note that because the number of Kostka numbers is
exponential in the input bitlength, the result of Barvinok and Fomin together
with the \#P-completeness of Kostka numbers does not contradict the (widely
believed) conjecture that \#P-complete quantities cannot be computed in
polynomial time.

As already mentioned above, it is known since the seminal work of Dyer, Frieze
and Kannan~\cite{dyerRandomPolynomialtimeAlgorithm1991} that using a
\emph{randomized} algorithm, the volume of a polytope in \(\R^n\) can be
approximated to a \((1\pm\epsilon)\) multiplicative factor in time polynomial in
\(1/\epsilon\) and the description length of the polytope, but the problem of
efficiently approximating the volume of polytopes using a \emph{deterministic}
algorithm remains open, notwithstanding a long line of work marked by several
interesting ideas.  We refer to the papers
\cite{barvinokComputingVolumeCounting1993} and \cite{BR22_quick} for references.
The latter paper, by Barvinok and Rudelson, finds the volumes of codimension
$\ell$ sections of an orthant to within a multiplicative factor of
$\exp(O(\ell))$.
Unfortunately, this result would not allow us to provably compute
the volume of an arbitrary Kostka polytope to within a sub-exponential
multiplicative factor in the dimension, because when one expresses the Kostka
polytope \(GT(\lambda, \mu)\) as a codimension \(\ell\) section of the orthant, one gets
$\ell = \Theta(n^2)$, which is of the same order as the dimension.  Here, we recall
again that previous works such as
\cite{BR22_quick,barvinokAsymptoticEstimatesNumber2009,barvinokMaximumEntropyGaussian2010}
focus not just on designing a deterministic volume approximation algorithm, but
also on providing a fairly explicit formula, perhaps involving the solution to an
optimization problem, that gives a good estimate of the volume.  Our result
also fits in this framework.

Questions concerning the positivity and more generally the computation of
representation theoretic multiplicities \cite{Panova1, Pak_Ikenmeyer} such as
Kostka numbers, Littlewood-Richardson coefficients \cite{BI, GCT3}, Kronecker
coefficients \cite{IMW17} and Plethysm coefficients \cite{IP17, FI20} have
arisen in the recent past in the context of Geometric Complexity Theory
\cite{GCT}.  It was shown that certain Lie theoretic multiplicities (when the
dimension of the Lie group is constant) could be computed exactly in polynomial
time in \cite{christandl}.  There has also been a significant amount of interest
in asymptotic algebraic combinatorics, which studies the asymptotic behavior of
these multiplicities; see \cite{Panova2} for a comprehensive survey.  Our
result, by providing a deterministic algorithm to approximate the volumes of
Kostka polytopes, which are semiclassical limits of Kostka numbers, also relates
to this body of work.

We conclude this section with a brief account of the discrete analog of
deterministic volume approximation: deterministic approximate counting of
discrete structures.  One of the first problems to be studied in this context
was the deterministic polynomial-time approximation of the permanent of
non-negative matrices~\cite{linialDeterministicStronglyPolynomial2000}.  Here,
the best known approximation factor, for an \(n \times n\) matrix, is
\(\sqrt{2}^n\), due to \cite{anariTightAnalysisBethe2021}, building upon
\cite{gurvitsBoundsPermanentApplications2014} and
\cite{vontobelBethePermanentNonnegative2013}.  Deterministic approximation based
on convex optimization has also been applied to the problem of approximating
permanents of positive semi-definite matrices (potentially with signed entries),
and approximation factors of the form \(c^n\) have also been obtained for that
case~\cite{anariSimplyExponentialApproximation2017,yuanMaximizingProductsLinear2022,ebrahimnejadApproximabilityPermanentPSD2024}:
note that it is not \emph{a priori} clear how to obtain any randomized Markov
chain Monte Carlo algorithms for permanents when entries can be negative
(but see also \cite{barvinokRemarkApproximatingPermanents2021} where such a
randomized algorithm is obtained for a special case via an integral identity).

Another related line of work in discrete approximate counting has applied
methods such as decay of correlations~\cite{bandyopadhyay_counting_2008,Weitz},
polynomial interpolation~\cite{barvinok2017combinatorics}, and the ``linear
programming'' method~\cite{moitraApproximateCountingLovasz2019} to obtain fully
polynomial time approximation schemes for the approximation of partition
functions.  More recently, some of these methods have also been applied to the
problem of volume
approximation~\cite{gamarnikComputingVolumeRestricted2023,bencsApproximatingVolumeTruncated2024,guoDeterministicApproximationVolume2025}
for well-studied polytopes such as the independent set and matching polytopes of
graphs: however, so far, this progress appears to be restricted to truncated
versions of these polytopes.

\subsection{Technical overview and organization}
We start in \cref{sec:defin-prel} with definitions of the Kostka and Schur-Horn
polytopes and some associated objects, including projected versions of these
polytopes.  We also collect in that section various important results and
observations.  The main technical content of the paper then follows, and is
divided into three main sections, which we now proceed to describe.

\paragraph{\textbf{Lower bound on the volume.}} Our first step, carried out in
\cref{sec-cont-anal-schur}, is to develop a good lower bound on the volume of
the Kostka polytope \(\GT(\lambda, \mu)\).  To do so, we introduce the main technical
ingredient in our approach: a continuous analogue of the Schur polynomial that
has been studied, e.g., in
\cite{shatashvili_correlation_1993,khare_sign_2021,gorinGaussianFluctuationsProducts2022}.
It turns out to be important for our purposes that we can interpret this
continuous analogue \(S_{\lambda}(x)\) as a (parameterized) Gibbsian partition
function. Consequences of this fact underlie the proof of
\cref{prop:correct-mean}, the main result of \cref{sec-cont-anal-schur}, where
we show that for an appropriate choice \(x^{*} = x^{*}(\lambda, \mu)\) of \(x\), the
value of \(S_{\lambda}(x^{*})\) can be used to give a good lower bound on the volume
of \(\GT(\lambda,\mu)\).  In particular, because of the Gibbsian form of
\(S_{\lambda}\) (see \cref{lem:geom-quant}), this \(x^{*}\) is the solution to a
convex optimization problem: a fact which is of vital importance when we later
turn to computation.  Another ingredient in the proof of
\cref{prop:correct-mean} is the log-concavity of the volume of
\(\GT(\lambda, \mu)\) when considered as a function of \(\mu\)
(~\cref{prop:log-concavity-of-vol}), combined with a result of Fradelizi that
relates the value of a log-concave density at its mean to the value at its mode.
A more detailed summary of this part of the paper appears at the beginning of
\cref{sec-cont-anal-schur}.

\paragraph{\textbf{Upper bound on the volume.}} In \cref{sec-geom-pro}, we complement the
results of \cref{sec-cont-anal-schur} by providing an upper bound on the volume
of \(\GT(\lambda, \mu)\).  The arguments here depend upon the geometric properties of
Kostka polytopes, and the main question turns out to be to understand how the
volume of \(\GT(\lambda, \mu)\) changes as \(\mu\) is perturbed.  This is analyzed in
\cref{lem:2.12}.  Towards this, an important ingredient is to establish that the
polytope contains a point that is somewhat far, in terms of certain conditioning
properties of the pair \((\lambda , \mu)\) (such as the minimum gap between consecutive
entries of \(\lambda\), and how far \(\mu\) is from the boundary of the permutohedron of
\(\lambda\)), from the bounding hyperplanes of the polytope (see
\cref{lem:2.10,prop-pgt-ball}).

\paragraph{\textbf{Computing the estimate.}} Combining the upper and lower bounds, we
obtain our main result, \cref{thm:main}.  It then remains to show how to
efficiently compute the estimate given by \cref{thm:main}.  The upper and lower
bounds on \(\GT(\lambda, \mu)\) obtained in \cref{sec-cont-anal-schur,sec-geom-pro} can
be expressed in terms of the value of \(S_{\lambda}\) at the point \(x^*\) mentioned
above.  As also mentioned above, this \(x^{*}\) is the solution to a convex
optimization problem.  However, algorithmically solving this convex optimization
problem requires an evaluation oracle for \(S_{\lambda}\).

To get such an oracle, we use another important property of \(S_{\lambda}\): for given
parameter \(x\), the continuous analogue \(S_{\lambda}(x)\) can be expressed as a
ratio of two determinants: see \cref{eq:s-det-unequal}.  However, two
difficulties arise in directly applying this determinantal formula.  The first
difficulty is that the formula is valid only when \(x\) has \emph{distinct}
entries.  This, however, can be easily handled, e.g., by considering carefully
chosen perturbations of \(x\): see the proof of \cref{prop-compute-g} for
details.  The second difficulty appears to be more fundamental: this is the fact
that the determinants in the formula are of matrices that have entries of the
form \(\exp(\lambda_ix_j)\) (where \(\lambda_i\) is an entry of \(\lambda\), and
\(x_j\) is an entry of \(x\)).  Note that the bit complexity of writing down a
good rational approximation to \(\exp(\lambda)\), even when \(\lambda\) is an integer, can
be exponential in the binary representation length of \(\lambda\).  It is mostly
because of this issue that our algorithm has a polynomial dependence on the
length \(\abs{\lambda}\) of \(\lambda\), rather than on its binary representation length.

Equipped with an approximate evaluation oracle for \(S_{\lambda}\)
(\cref{prop-compute-g}), we are then able to apply standard results on convex
optimization with a weak membership oracle~\cite[see \cref{thm-gls-optimizer}
below]{YN76,GLSbook} to perform the evaluation of the volume estimate.
 
\subsection{Discussion}

The Schur-Horn problem (i.e., when does there exist a Hermitian matrix with
prescribed spectrum and diagonal) can be seen as a special case of the more
general question of understanding the orbits of certain actions of a Lie group
on elements of its Lie algebra (in the case of the Schur-Horn problem, one is
looking at the conjugation action of the unitary group on real diagonal
matrices).  The latter question also has connection to marginal consistency
problems in quantum information theory~\cite{Collins-McSwiggen}. The question of approximating the density of
the pushforward measure efficiently can be asked for compact groups other than
the unitary group.

Another related open question is to deterministically approximate, within a multiplicative factor of $\exp(o(n^2))$,  the Horn density function at a given spectrum in $\R^n$. This function measures the probability density that the spectrum of the sum of two independent random Hermitian matrices from unitarily invariant distributions,  whose spectra are $\la$ and $\mu$ respectively,  equals $\nu$, where $\la, \mu$ and $\nu$ are in $\R^n$ with non-decreasing coordinates. This density can be expressed in terms of the volumes of the hive polytopes (of dimension $\sim \frac{n^2}{2}$) introduced by Knutson and Tao \cite{honey, Coquereaux2020}.

\section{Definitions and Preliminaries}
\label{sec:defin-prel}
We use $c, C$ etc to denote universal constants whose value may differ from
occurrence to occurrence.  Throughout the paper, we assume that the parameter
$n$ is an integer that is at least $3$, unless otherwise specified.  By a slight
abuse of a usual notation for partitions, we will denote the sum of the entries
of any finite length vector \(v\) by \(\abs{v}\).  We will denote by
\(\norm{v}\) the standard Euclidean norm of such a vector.

\subsection{Kostka and Schur-Horn polytopes}

\begin{definition}[\textbf{Partitions}] A \emph{partition}
  $\lambda = (\lambda_1, \lambda_2, \dots, \lambda_k)$ is a non-increasing
  tuple of non-negative real numbers. Such a $\lambda$ is said to be a
  \emph{partition of $\abs{\lambda} \defeq \sum_{i=1}^k \lambda_i$ into $k$
    parts}.  A partition is said to be \emph{integral} if its entries are
  integers.  Given a non-negative integer $m$, we further denote by
  $\opart{m}{n}$ the set of all integral partitions of $m$ into $n$ parts.
\end{definition}

We denote by $s_\lambda(x_1, x_2, \dots, x_n)$ the generating polynomial (known
as the \emph{Schur polynomial}) of semi-standard Young tableaux with shape $\lambda$ (where $\lambda$
is an integral partition) and with entries in $[n]$.  Thus,
$s_\lambda(x_1, x_2, \dots, x_n)$ is a homogeneous polynomial of degree
$\abs{\lambda}$, and can be written as
\begin{equation}
  \label{eq-schur-poly}
  s_\lambda(x_1, x_2, \dots x_n) \defeq \sum_{\mu \in \Z_{\geq 0}^{n} \st \abs{\mu} = n}
  \kostka{\lambda}{\mu} x^\mu,
\end{equation}
where the \emph{Kostka number} $\kostka{\lambda}{\mu}$ is the number of semi-standard
Young tableaux with shape $\lambda$ and type $\mu$.\footnote{Recall that a
  \emph{semi-standard Young tableau} with \emph{shape} \(\lambda\) is a left-aligned,
  two-dimensional ragged array of integers in which the \(i\)th row from the top
  has length \(\lambda_i\), and such that the integers along each row are weakly
  increasing from left to right, while the integers along each column are
  strictly increasing from top to bottom.  Such a tableau is said to have
  \emph{type} (or \emph{weight}) \(\mu\) if the total number of times any integer
  \(k\) appears in the tableau is exactly equal to \(\mu_k\).  See, e.g.,
  \cite[p.~309]{Stanley1999}.} The Schur polynomials are efficiently computable
due to the classical Jacobi-Trudi identity, which we now describe.  Let
$h_{\ell}(x_1, x_2, \dots, x_n)$ denote the complete symmetric polynomial of degree
$\ell$ in the variables $x_1, x_2, \dots, x_n$: thus
$h_{\ell}(x_1, x_2, \dots, x_n)$ is the sum of all $\binom{n + \ell - 1}{\ell}$ monomials
of degree $\ell$ in the variables $x_1, x_2, \dots, x_n$.  The Jacobi-Trudi
identity is then given by~(see, e.g.,~\cite[p.~75]{fulton_1996}):
\begin{equation}
  s_{\lambda}(x_1, x_2, \dots, x_n) = \det\insq{h_{\lambda_i + j - i}(x_1, x_2,
    \dots, x_n)}_{1 \leq i, j \leq n}.
\end{equation}
The Jacobi-Trudi identity can also be written as (see,
e.g.,~\cite[p.~75]{fulton_1996}):
\begin{equation}
  \label{eq-jacobi-trudi-ratio-of-dets}
  s_{\lambda}(x_1, x_2, \dots, x_n) \cdot V(x_1, x_2, \dots, x_n)
  =
  \det\insq{x_i^{\lambda_j + n - j}}_{1 \leq i, j \leq n},
\end{equation}
where
$V(x_1, x_2, \dots, x_n) \defeq \prod_{1 \leq i < j \leq n}(x_i -
x_j)$ is a Vandermonde determinant.

\begin{definition}[\textbf{Kostka/Gelfand-Tsetlin (GT) polytope (see,
  e.g.,~\cite{Alexandersson2020-GT} or \cite[Definitions 1.7 and 1.10]{baryshnikov_gues_2001})}]\label{def-kostka-polytope}
A \emph{Gelfand-Tsetlin (GT)} pattern with $n$ rows is a triangular array
$\vec{x} = \inp{x_{i\,j}}_{\substack{1 \leq i \leq n\\ 1 \leq j \leq i}}$ of
real numbers satisfying
  \begin{gather}
x_{i+1\, j}\geq x_{i\, j} \geq x_{i+1\, j + 1}
    \text{ for all } 1 \leq i \leq n - 1 \text{ and } 1 \leq j \leq i.
    \label{eq:gt-cons}
  \end{gather}
  Denote by \(x_{{i\, \cdot}}\) the \(i\)th row of the pattern, so that
  $\abs{x_{i\, \cdot}}$ denotes the sum $\sum_{1 \leq j \leq i}x_{i\,j}$.  Note that if a row
  \(x_{j\, \cdot}\) of a GT-pattern is non-negative, then each lower row
  $x_{i\, \cdot}$ for $i < j$ is also non-negative, and satisfies
  \(\abs{x_{i\, \cdot}} \leq \abs{x_{j\,\cdot}}\).

The set of all \GT{}
  patterns with $n$ rows is denoted as $\GT_n$, and forms a convex cone in
  $\R^{n(n+1)/2}$.  Given a non-increasing $n$-tuple $\lambda$ of real numbers, we
  denote by $\GT_n(\lambda)$ (and often just by $\GT(\lambda)$) the section of this cone of
  \GT{} patterns in which the top row is fixed to be $\lambda$, i.e., in which
  $x_{n,j} = \lambda_{j}$ for $1 \leq j \leq n$.  For technical reasons, it is convenient to
  define the set of points in $\GT(\lambda)$ with the top row removed:
  \begin{equation}
    \label{eq:17}
    \tpGT(\lambda) \defeq \inb{
      \inp{x_{i\,j}}_{
        \substack{
          1 \leq i \leq n - 1\\
          1 \leq j \leq i}
      } \st \vec{x} \in \GT(\lambda)} \subseteq \R^{n(n-1)/2}.
  \end{equation}
  We will denote the associated map from $\GT(\la)$ to $\tpGT(\la)$ by $p$.
  The volume of $\GT(\lambda)$ (according to the $n(n-1)/2$-dimensional
  Hausdorff measure in $\R^{n(n+1)/2}$) is then given by~(see, e.g.,~\cite[Lemma
  1.12]{baryshnikov_gues_2001}):
  \begin{equation}
    \label{eq:gt-volume}
    \H^{n(n-1)/2}\inp{GT(\lambda)} = \lvol{\tpGT(\lambda)}
    =  \prod_{1 \leq i < j \leq n}\frac{\lambda_i -
      \lambda_j}{j - i}.
  \end{equation}

Given two $n$-tuples $\lambda$ and $\mu$ of real numbers, where $\lambda$ is
  non-increasing, $\GT_n(\lambda, \mu)$ (usually written just as
  $\GT(\lambda, \mu)$) denotes the section of $\GT(\lambda)$ obtained by
  imposing the following further equalities:
  \begin{align}
    x_{11}
    & = \mu_1, \text{ and } \label{eq:gt-section-first} \\
    \abs{x_{i\, \cdot}} - \abs{x_{i - 1\, \cdot}}
    & = \mu_i
      \text{ for }
      2 \leq i \leq n.\label{eq:gt-section}
  \end{align}
  (Recall that $\abs{x_{k\, \cdot}}$ denotes the sum $\sum_{1 \leq j \leq k}x_{k\, j}$.)
  Again for technical reasons, it is sometimes convenient to consider the sets
  \begin{align}
    \tpGT(\lambda, \mu)
    &= \inb{\inp{x_{i\,j}}_{\substack{1 \leq i \leq n -
      1\\1\leq j \leq i}} \st \vec{x} \in \GT(\lambda, \mu)}
    \subseteq \R^{n(n-1)/2}\text{, and} \\
    \pGT(\lambda, \mu)
    &= \inb{\inp{x_{i\,j}}_{\substack{2 \leq i \leq n -
      1\\1\leq j \leq i - 1}} \st \vec{x} \in \GT(\lambda, \mu)}
    \subseteq \R^{(n-1)(n-2)/2}.
  \end{align}
  We will denote the associated map from $\GT(\lambda, \mu)$ to
  $\pGT(\lambda, \mu)$ as $\tp$.  From this, it is easy to write down the
  defining linear inequalities for $\pGT(\lambda, \mu)$, which we do below.
  \begin{gather}
    \sum_{k=1}^{n-1}\mu_k - \lambda_{n-1} \leq \sum_{j=1}^{n-2} x_{n-1\,j}
    \leq \sum_{k=1}^{n-1}\mu_k - \lambda_n,\label{eq:3}\\
    \lambda_{j+1} \leq x_{n-1\, j} \leq \lambda_j
    \text { for } 1 \leq j \leq n - 2,
    \label{eq:5}\\
    x_{i+1\,j} \leq x_{i\,j} \leq x_{i+1\,j}
    \text{ for } 2 \leq i \leq n -2
    \text{ and } 1 \leq j \leq i - 1,
    \label{eq:6}\\
    \sum_{j=1}^i x_{i+1\,j} - \sum_{j=1}^{i-1}x_{i\,j} \geq \mu_{i+1}
    \text{
      for } 1 \leq i \leq n - 2, \text{ and }
    \label{eq:7}\\
    x_{i+1\,i} + \sum_{j=1}^{i-1}x_{i\,j} \geq
    \sum_{k=1}^i\mu_k \text{ for } 1 \leq i \leq n - 2.
    \label{eq:10}
  \end{gather}
  We record a simple consequence of these inequalities.  \(\SH(\lambda)\) in the
  following proposition denotes the \emph{Schur-Horn} polytope of \(\lambda\) (see
  \cref{def-schur-horn-polytope}).  In particular, it is known that
  \(\GT(\lambda, \mu)\) is non-empty if and only if \(\mu \in \SH(\lambda)\).
  \begin{proposition}
    \label{prop-hyperplane-move}
    Let $\lambda$ be a partition with $n$ parts and suppose that $\mu \in
    \SH(\lambda)$. If $\mu'$ is such that
    $\norm{\mu - \mu'} \leq \delta$, then, in going from $\pGT(\lambda, \mu)$ to
    $\pGT(\lambda, \mu')$, each of the bounding hyperplanes of
    $\pGT(\lambda, \mu)$ is translated perpendicular to itself by a distance of
    at most $\frac{\sqrt{n-1}}{\sqrt{n-2}}\delta$.
  \end{proposition}
  \begin{remark}
    An immediate consequence of the above proposition is that if an
    $\binom{n-1}{2}$-dimensional Euclidean ball of radius
    $r >\frac{\sqrt{n-1}}{\sqrt{n-2}}\delta$ centered at a point $z$ is
    contained in $\pGT(\lambda, \mu)$, then an $\binom{n-1}{2}$-dimensional
    Euclidean ball of radius $r - \frac{\sqrt{n-1}}{\sqrt{n-2}}\delta$ centered
    at the same point $z$ is contained in $\pGT(\lambda, \mu')$.
  \end{remark}
  \begin{proof}[Proof of \cref{prop-hyperplane-move}]
    The coefficients of the variables in the defining constraints
    (\cref{eq:3,eq:5,eq:6,eq:7,eq:10}) of \(\pGT(\lambda, \mu)\) do not depend upon
    \(\mu\).  It thus follows that in modifying \(\mu\) to \(\mu'\), each of the
    defining hyperplanes is translated perpendicular to itself.  Further, hyperplanes
    corresponding to \cref{eq:5,eq:6} do not change.  The right hand sides of
    the inequalities in \cref{eq:3} change by at most $\delta\sqrt{n-1}$, and the
    Euclidean norm of the coefficient vector for the corresponding hyperplanes
    is $\sqrt{n-2}$, so that the corresponding hyperplanes translate along their
    normal by a distance of at most $\frac{\sqrt{n-1}}{\sqrt{n-2}}\delta$.  A similar
    argument applied to \cref{eq:7} and \cref{eq:10} gives upper bounds of
    $\delta/\sqrt{2i-1}$ (for \(1 \leq i \leq n - 2\)) and $\delta$ respectively on the
    magnitude of the translation, so that the largest possible magnitude of the
    translation over all the defining hyperplanes is
    $\frac{\sqrt{n-1}}{\sqrt{n-2}}\delta$.
  \end{proof}

  The {volume} of $\GT(\lambda, \mu)$ according to the
  $\frac{(n-1)(n-2)}{2}$-dimensional Hausdorff measure in $\R^{n(n-1)/2}$ is
  denoted by $V_{\lambda, \mu}$, while that of $\pGT(\lambda, \mu)$ with respect
  to the Lebesgue measure on $\R^{(n-1)(n-2)/2}$ is denoted by
  $\tV_{\lambda, \mu}$.  It is then a simple consequence of the area
  formula~\cite[Theorem 3.2.3]{F96} that
  \begin{equation}
    \label{eq:vol-relation}
    \sqrt{(n-1)!}\cdot \tV_{\lambda,\mu} = V_{\lambda, \mu}.
  \end{equation}
  (See \cref{sec:comp-cref-vol} for details).  Note that for these volumes to be
  non-zero, it is necessary that (1) $\abs{\lambda} = \abs{\mu}$, and (2) the
  entries of $\lambda$ must be distinct.

  When both $\lambda$ and $\mu$ have only non-negative integer entries, integer points
  in $\GT(\lambda, \mu)$ are in one-to-one correspondence with semi-standard Young
  tableaux with shape $\lambda$ and type $\mu$.  (This correspondence is obtained by
  interpreting the $i$th row of a pattern with integral entries as the shape of
  the sub-tableaux consisting only of entries in $\inb{1, 2, \dots, i}$.)  In
  particular, the number of integral points in $\GT(\lambda, \mu)$ is then equal to the
  Kostka number $\kostka{\lambda}{\mu}$.  We thus refer to the polytope
  $\GT(\lambda, \mu)$ as a \emph{Kostka} polytope.
\end{definition}
\begin{definition}[\textbf{The Schur-Horn polytope}]\label{def-schur-horn-polytope}
  Let $\lambda$ be a non-increasing \(n\)-tuple of real numbers.  For $\sigma$ a
  permutation of $\{1, \dots, n\}$, we denote its action on $\la$ by
  \begin{equation}
    \sigma(\la_1, \dots, \la_n) := (\la_{\sigma(1)}, \dots, \la_{\sigma(n)}).
  \end{equation}
  We refer to the convex hull of all vectors $\sigma(\lambda)$ (as $\sigma$
  ranges over all $n$-permutations)
as the \emph{Schur-Horn} polytope $\SH(\la)$.  $\SH(\lambda)$ is also referred
  to as the \emph{permutohedron} of $\lambda$~\cite{Post09}. It is well known
  that $\SH(\lambda)$ has the following equivalent representation~\cite[Theorem
  1.1 and Remark 1.1]{M64}.
  \begin{equation}
    \SH(\lambda) = \inb{\mu \in \R_{\geq 0}^n \st \lambda \succeq \mu},
  \end{equation}
  where $\lambda \succeq \mu$ means that $\lambda$ \emph{majorizes} $\mu$ in the
  sense that for every permutation\footnote{It is enough to restrict to the
    permutation that puts the entries of $\mu$ in non-increasing order.} $\tau$
  of $\inb{1, 2, \dots, n}$,
  \begin{equation}
    \label{eq:majorization}
    \begin{aligned}[c]
      \sum\limits_{j=1}^{i}(\lambda_j - \mu_{\tau(j)})
      &\geq 0 \text{ for } 1 \leq i \leq n - 1,  \text{ and}
      \\
      \sum\limits_{j=1}^{n}(\lambda_j - \mu_j)
      &= 0.
    \end{aligned}
  \end{equation}
  Since majorization is a partial order on partitions, it follows that if
  $\lambda$ and $\mu$ are partitions such that $\lambda \succeq \mu$, then
  $\SH(\lambda) \supseteq \SH(\mu)$.

  Again, it is sometimes convenient to consider the set $\pSH(\lambda)$ defined
  as
  \begin{equation}
    \label{eq:22}
    \pSH(\lambda) \defeq \inb{p(\mu) \defeq (\mu_1, \mu_2, \dots, \mu_{n-1})
      \st
      \mu \in \SH(\lambda)
    } \subseteq \R^{n-1}.
  \end{equation}
  Here we abuse notation slightly to denote the canonical map from
  $\SH(\lambda)$ to $\pSH(\lambda)$ also by $p$ (the same notation was used for
  the map from $\GT(\lambda)$ to $\tpGT(\lambda)$ in \cref{eq:17}). We also define
  the inverse map $p_{\lambda}^{-1}: \pSH(\lambda) \rightarrow \SH(\lambda)$ as
  \begin{equation}
    \label{eq:63}
    p_{\lambda}^{-1}(\tilde{\nu}) \defeq (\tilde{\nu}_1, \tilde{\nu}_2, \dots,
    \tilde{\nu}_{n-1}, \abs{\lambda} - \abs{\tilde{\nu}}).
  \end{equation}
  A consequence of the area formula~\cite[Theorem 3.2.3]{F96} is the following
  equality between the $(n-1)$-dimensional Hausdorff volume of $\SH(\lambda)$
  and the volume of $\pSH(\lambda)$ (see \cref{sec:comp-cref-vol} for details):
  \begin{equation}
    \label{eq:schul-vol-rel}
    \H^{n-1}(\SH(\lambda)) = \sqrt{n}\cdot \lvol{\pSH(\lambda)}.
  \end{equation}
\end{definition}

The quantity $\lvol{\pSH(\lambda)}$ has been studied by Postnikov~\cite[Theorem
3.1]{Post09}.\footnote{Note that Postnikov defines the volume of $\SH(\lambda)$
  to be the volume of $\pSH(\lambda)$: see \cite[p.~1030]{Post09}.}  We record
the following corollary of Postnikov's formula.
\begin{proposition}[\textbf{Direct corollary of Theorem 3.1 of \cite{Post09}}]
  \label{prop-vol-psh} Let $\lambda$ be a partition with $n$ parts.  Then
  $\lvol{\pSH(\lambda)} \leq \lambda_1^{n-1}n^{2n}$.
\end{proposition}
\begin{proof}
  A special case of Theorem 3.1 of \cite{Post09} gives the following equality:
  \begin{equation}
    \lvol{\pSH(\lambda)}
    = \frac{1}{(n-1)!}
    \sum_{\sigma \in \mathrm{Sym}_n}\frac{
      \inp{\sum_{i=1}^n\lambda_i\sigma(i)}^{n-1}
    }{
      \prod_{j=1}^{n-1}(\sigma(j) - \sigma(j+1))
    }.
  \end{equation}
  The claim now follows by upper bounding the absolute value of each term by
  $\inp{\frac{n(n+1)\lambda_1}{2}}^{n-1}$ which is at most
  $n^{2n-1}\lambda_1^{n-1}$ for $n \geq 1$.
\end{proof}

\paragraph{\textbf{Spectra and the Schur-Horn and Gelfand-Tsetlin polytopes.}} The
\emph{Schur-Horn} theorem says that given an $n \times n$ Hermitian matrix $H$ with
spectrum $\lambda$ (arranged in non-increasing order), the set of diagonals of all
$n \times n$ matrices unitarily equivalent to $H$ is exactly $\SH(\lambda)$.  Following the
notation of \cite{LMV21}, we denote by $R(H)$ the \emph{Rayleigh
  map}~\cite{Ner03} which maps an $n \times n$ Hermitian matrix $H$ with spectrum
$\lambda$ as above to a triangular array whose $i$th row, for
\(1 \leq i \leq n\), is the spectrum of the $i \times i$ leading principle submatrix of
$H$.  The Cauchy interlacing theorem then implies that
$R(UHU^{\dagger}) \in \GT(\lambda)$ for any unitary $U$.  It then follows that if we define
the map
\begin{equation}
  \label{eq-alpha-p}
  \alpha': \GT(\lambda) \rightarrow \R^{n}
  \text { given by }
  \alpha'(z)_i = \abs{z_{i, \cdot}} - \abs{z_{i-1,\cdot}}
  \text{ for }
  1 \leq i \leq n,
\end{equation}
(where we assume the convention $\abs{z_{0, \cdot}} = 0$) then
$\alpha'(R(UHU^{\dagger}))$ is the diagonal of $UHU^{\dagger}$.  Thus,
$\SH(\lambda)$ is a subset of the range of $\alpha'$, and it can be shown that the range is
in fact $\SH(\lambda)$ (see, e.g., \cite[Section 2]{LMV21}).  One also sees from this
correspondence that the polytope $\GT(\lambda, \mu)$ is non-empty if and only if
$\lambda \succeq \mu$.

\subsection{Log-concavity of volumes of Kostka polytopes}
The following is a special case of a well known general idea (see,
e.g.,~\cite{okounkov_why_2003}) that the Brunn-Minkowski inequality implies that
volumes of polytopes parameterized linearly are log concave in the
parameterization.
\begin{proposition}[\textbf{Log concavity of $V_{\lambda,\mu}$}]
  \label{prop:log-concavity-of-vol} Fix a partition $\lambda$ with $n$ parts.
  Then for any $a \in (0, 1)$ and non-negative $n$-tuples $\mu_1$ and $\mu_2$,
  \begin{equation}
    V_{\lambda, a\mu_1 + (1-a)\mu_2} \geq V_{\lambda,\mu_1}^a\cdot V_{\lambda, \mu_2}^{1-a}.
  \end{equation}
\end{proposition}
\begin{proof}
  Consider points $x$ and $y$ in $\pGT(\lambda, \mu_1)$ and
  $\pGT(\lambda, \mu_2)$ respectively.  Then, an inspection of the constraints
  in \cref{eq:gt-cons,eq:gt-section-first,eq:gt-section} shows
that $ax + (1-a)y$ is a point in $\pGT(\lambda, a\mu_1 + (1-a)\mu_{2})$.
  Thus, for the indicator functions
  $\ind{\pGT(\lambda, a\mu_1 + (1-a)\mu_{2})}$, $\ind{\pGT(\lambda, \mu_1)}$ and
  $\ind{\pGT(\lambda, \mu_2)}$, we have that for every $x, y$
  \begin{equation}
    \ind{\pGT(\lambda, a\mu_1 + (1-a)\mu_{2})}(ax + (1-a)y)
    \geq
    \ind{\pGT(\lambda, \mu_1)}(x)^a
    \cdot \ind{\pGT(\lambda, \mu_2)}(y)^{1-a}.
  \end{equation}
  Applying the \pl{} inequality for the Lebesgue measure in $\R^{(n-1)(n-2)/2}$
  now yields
  \begin{equation}
    \tV_{\lambda, a\mu_1 + (1-a)\mu_2} \geq \tV_{\lambda,\mu_1}^a\cdot \tV_{\lambda, \mu_2}^{1-a}.
  \end{equation}
  The claim now follows since, from \cref{eq:vol-relation}, the ratio between
  $\tV_{\lambda, \mu}$ and $V_{\lambda, \mu}$ does not depend upon $\mu$.
\end{proof}

As an immediate corollary, we get the following.
\begin{corollary}\label{cor:log-concave-f}
  Let $\lambda$ be a partition with $n$ distinct parts.  Fix
  $w \in \R_{\geq 0}^n$.  Then, the function
  \begin{equation}
    \label{eq:f-def}
    f_{\lambda, w}(\nu) \defeq V_{\lambda,\nu}\exp(w\cdot \nu)
  \end{equation}
  is log-concave.  Note that $f_{\lambda, w}(\nu) = 0$ if $\nu$ is not in
  $\SH(\lambda)$.
\end{corollary}

\subsection{Log-concave densities}
We will need the following result of Fradelizi~\cite{fradelizi_sections_1997}.
\begin{theorem}[\textbf{\cite[Theorem 4]{fradelizi_sections_1997}}] Let
  \(K \subseteq \R^{N}\) be a convex body and let $f: \R^{N} \rightarrow \R$ be a non-negative
  valued log-concave function supported on \(K\) such that
  $\int\limits_{x \in \R^N}f(x) dx$ is positive.  Define $\mean{f}$ to be the mean of
  the probability distribution with density proportional to $f$. Thus,
  \begin{equation}
    \label{eq-mean-f}
    \mean{f} \defeq \frac{\int\limits_{x \in \R^N}xf(x) dx}{\int\limits_{x \in \R^N}f(x)
      dx}.
  \end{equation}
  Then,
  \begin{equation}
    f(v) \leq \exp(N) \cdot f(\mean{f}) \text{ for all } v \in \R^N.
  \end{equation}\label{them:fradelizi}
\end{theorem}

\begin{table}[t]
  \centering
  \begin{tabular}{|c|l|}
    \hline
    Notation
    & Location\\
    \hline
    $\gap$
    & \Cref{def-gap}\\
    \hline
    $\tau(\la, \mu)$
    & \Cref{eq:tau-def}\\
    \hline
    $r(\lambda, \mu)$
    & \Cref{eq-r-def}\\
    \hline
    $\delta'(\lambda, \mu)$
    & \Cref{eq-delta-prime-def}\\
    \hline
    $r_0(\lambda)$
    & \Cref{eq-r0}\\
    \hline
    $\bdrad(\lambda, \mu)$
    & \Cref{eq-bdrad-def}\\
    \hline
  \end{tabular}
  \caption{Notation index}
  \label{tab:notation}
\end{table}

\section{Lower bounding the volume: Continuous analogues of Schur polynomials}
\label{sec-cont-anal-schur}
In this section, we derive a lower bound on the volume of the Kostka polytope
\(\GT(\lambda, \mu)\).  Our main technical tool for this is a continuous analogue of the
(discrete) Schur polynomials discussed in the introduction.  This continuous
analogue of the Schur polynomials can be defined in a few equivalent ways; for
our purposes, it will be convenient to use the following definition, which has
been used, e.g. by Khare and Tao~\cite[eqs.~(5.10, 5.11)]{khare_sign_2021}.  Our
notation however is closer to the exposition by Tao in \cite{Tao-blog}, and we
denote the continuous analogue of the Schur polynomials $s_\la(x)$ by
$S_\la(x)$, where $\lambda \in \R^n$ is a non-increasing tuple, and $x \in \R^n$ is an
arbitrary vector. The function $S_{\lambda}(x)$ is then defined by
\begin{equation}
  \label{eq:tao-S-def}
  S_{\lambda}(x)
  \defeq
  \int\limits_{z \in \GT(\lambda)}
  \exp\inp{
    \sum_{i = 1}^{n}
    \inp{\abs{z_{i, \cdot}} - \abs{z_{i-1,\cdot}}}x_i
  }
  \quad
  \prod_{\substack{1 \leq i \leq {n-1}\\1 \leq j \leq i }} dz_{i,j},
\end{equation}
where $\abs{z_{0, \cdot}}$ is assumed by convention to be $0$.

For our purposes, three important properties of the continuous analogue
\(S_{\lambda}\) turn out to be of crucial both in the proof of the volume lower bound
in \cref{prop:correct-mean}, and also for later developments in the paper.  We
now discuss these in order.

{\interfootnotelinepenalty=10000
  \paragraph{\textbf{\(S_{\lambda}\) as a Gibbs partition function.}}  Note that
  \(S_{\lambda}(x)\), as defined in \cref{eq:tao-S-def}, can be seen as the partition
  function of a Gibbsian probability distribution on \(\GT(\lambda)\), with
  \(x\) playing the role of a vector-valued ``inverse temperature'' and
  \(\alpha'(z)\) (where the map \(\alpha'\) from \(\GT(\lambda)\) to
  \(\SH(\lambda)\) is as defined in \cref{eq-alpha-p}) playing the role of a
  vector-valued ``energy functional''.  The first important consequence of this
  general form of \(S_{\lambda}\) that matters for our purposes is that the gradient,
  with respect to \(x\), of \(\log S_{\lambda}(x)\) is \emph{exactly} the mean of
  \(\alpha'(z)\) when \(z\) is distributed in \(\GT(\lambda)\) according to a probability
  density proportional to the integrand in the definition of
  \(S_{\lambda}(x)\).\footnote{This is also related to the idea, going back to
    Boltzmann~\cite{Boltzmann1877}, that the integrand in the definition of
    \(S_{\lambda}\) is proportional to a probability density on \(\GT(\lambda)\) that has
    the maximum possible entropy under the constraint that the function \(\alpha'\) has a
    given mean.  However, this connection by itself will not be important for our
    purposes.}  }

\paragraph{\textbf{\(S_{\lambda}\) and the permutohedron \(\SH(\lambda)\).}} The second important
property of \(S_{\lambda}\) concerns projecting the integrand in its definition to a
measure on the permutohedron \(\SH(\lambda)\) of \(\lambda\). This leads to a probability
density on \(SH(\lambda)\) that
\begin{enumerate}
\item at any point \(\mu \in \SH(\lambda)\), is proportional to the volume of
  \(\GT(\lambda, \mu)\) multiplied by an exponential tilt (as we show in
  \cref{lem:geom-quant} below), and
\item is log-concave on \(SH(\lambda)\): this is a consequence of
  \cref{lem:geom-quant} below combined with \cref{cor:log-concave-f} proved
  above.
\end{enumerate}
Along with Fradelizi's bound~(\cref{them:fradelizi}), the above facts all play
a crucial role in the proof of the volume lower bound in
\cref{prop:correct-mean}.

\paragraph{\textbf{\(S_{\lambda}\) is a ratio of two determinants.}}
As observed in \cite[eqs.~(5.10, 5.11)]{khare_sign_2021}, an integration formula
due to Shatashvili~\cite[eq. (3.2)]{shatashvili_correlation_1993} implies the
following identity for $S_{\lambda}(x)$, which is valid when the entries of $x$ are
pairwise distinct:
\begin{equation}
  \label{eq:s-det-unequal}
  S_\la(x) = \frac{\det[\exp(x_i\la_j)]_{1\leq i,j \leq n}}{V(x)}.
\end{equation}
Here, $V(x) \defeq \prod_{1 \leq i < j \leq n}(x_i - x_j)$ is a Vandermonde determinant.
It is directly apparent from this representation (and the continuity of
$S_{\lambda}$) that $S_{\lambda}(x)$ is symmetric in \(x\).  This representation in terms of
determinants does not play much role in this section, but is crucial when we
consider the computational aspects of our estimates, in \cref{sec:comp-aspects}.

Before proceeding with the technical content of this section, we also record
a basic algebraic property of \(S_{\lambda}(x)\):  it follows directly from
the definition of $S_{\lambda}$ and the scaling and translation properties of
$\GT(\lambda)$ that for any positive real numbers $\alpha$ and $\beta$,
\begin{equation}
  \label{eq-scale-shift-s}
  \begin{aligned}
    S_{\lambda + \alpha \one}(x)
    &= \exp(\alpha\one^{\top}x)S_{\lambda}(x) \text{, and}\\
    S_{\lambda}(\beta x)
    &= \beta^{-\binom{n}{2}}S_{\beta\lambda}(x).
  \end{aligned}
\end{equation}

We now begin with the following simple observation.

\begin{proposition}
  Let $\la$ be a partition with distinct parts and let $\mu$ be a $n$-tuple of
  non-negative reals such that \(\mu\) is in the relative interior of
  \(\SH(\lambda)\).  Then the function
  \begin{equation}
    \label{eq:38}
    g_{\lambda,\mu}(x) \defeq \log S_{\lambda}(x) - x^T\mu
  \end{equation}
  from $\R^n$ to
  $\R$ is convex. \label{prop:log-convex-cont-schur}
\end{proposition}
\begin{proof}
Since the second term is linear, it suffices to prove that $S_{\lambda}$ is
  log convex.  This is a consequence of H\"older's inequality: for any
  $\theta \in (0, 1)$ and $x, y \in \R^n$, we apply H\"older's inequality with
  norms $1/\theta$ and $1/(1-\theta)$ to get (using the definition of
  $S_\lambda$ in \cref{eq:tao-S-def})
  \begin{equation}
    S_{\lambda}(\theta  x + (1-\theta)y)
    \leq S_{\lambda}(x)^{\theta} \cdot S_{\lambda}(y)^{1-\theta}. \qedhere
  \end{equation}
\end{proof}

That one can view the continuous analogue of Schur polynomial as a
“semiclassical limit” of its discrete counterpart (\cref{eq-schur-poly}), in a
manner that can be made precise using the machinery of geometric quantisation,
is mentioned in \cite{Tao-blog}; we prove a related identity below.  The
identity below shows that \(S_{\lambda}(x)\) corresponds roughly to replacing,
in the definition of the Schur polynomial, Kostka numbers by volumes of Kostka
polytopes, and the sum by an integral.  It was pointed out to us by Colin
McSwiggen that this identity can be situated in the more general context of
Harish-Chandra integrals, and in particular, after appropriate translation of
notation, can be seen as a special case of equation (73) on p.~461 of
\cite{McSwiggen2021BoxSplines}.
\begin{lemma} Let $\lambda \in \R^n$ be a non-increasing tuple. Then
  \begin{equation}
    \label{eq:geom-quant}
    \begin{aligned}[c]
      S_\la(x)
      &=
        \frac{1}{\sqrt{(n-1)!}} \cdot
        \int\limits_{\nu \in \SH(\lambda)}
        V_{\la \nu} \exp(x \cdot \nu)
        \quad
        \prod\limits_{1 \leq i \leq n - 1} d\nu_{i}\\
      &\overset{\textup{\cref{eq:vol-relation}}}{=} \int
        \limits_{\nu\in\SH(\la)}
        \tV_{\la, \nu}
        \exp(x \cdot \nu)
        \prod\limits_{1 \leq i\leq n-1} d\nu_i.
    \end{aligned}
  \end{equation}\label{lem:geom-quant}
\end{lemma}
\begin{proof} We start with the definition of
  $S_{\lambda}(x)$ given in \cref{eq:tao-S-def}, and apply the co-area
  formula~\cite[Theorem 3.2.12]{F96} with the map
  \begin{equation}
    \label{eq:34}
    \alpha: \R^{n(n-1)/2} \rightarrow \R^{n-1}
    \text{ given by }
    \alpha(z)_i = \abs{z_{i, \cdot}} - \abs{z_{i-1,\cdot}}
    \text{ for }
    1 \leq i \leq n-1.
  \end{equation}
  Note that $\alpha$ is obtained from the onto map
  $\alpha': \GT(\lambda) \rightarrow \SH(\lambda)$ defined in \cref{eq-alpha-p}
  by simply ignoring the last coordinate, and it thus maps $z \in \tpGT(\lambda)$
  onto $\pSH(\lambda)$.  The co-area formula then gives (after a calculation of
  the $(n-1)$-dimensional Jacobian of $\alpha$: see below)
  \begin{multline}
    \label{eq:24}
    \int\limits_{z \in \GT(\lambda)}
    \exp\inp{
      \sum_{i = 1}^{n}
      \inp{\abs{z_{i, \cdot}} - \abs{z_{i-1,\cdot}}}x_i
    }
    \quad
    \prod_{\substack{1 \leq i \leq {n-1}\\1 \leq j \leq i }} dz_{i,j} \\
    =
    \frac{1}{\sqrt{(n-1)!}}
    \cdot \int\limits_{\nu \in \SH(\lambda)}\exp(x\cdot \nu) V_{\lambda,\nu}
    \quad
    \prod\limits_{1 \leq i \leq n - 1} d\nu_{i},
  \end{multline}
  which implies the claim of the lemma.  It only remains to complete the
  computation of the Jacobian required in order to apply the co-area formula.

  Let \(D\alpha(z)\) denote the derivative of the map \(\alpha\).  We then note that the
  co-area formula~\cite[Theorem 3.2.12]{F96}, implies \cref{eq:24} provided that
  the Jacobian
  \begin{equation}
    J_{\alpha}(\vec{z}) \defeq \sqrt{\det\inp{D\alpha(z)D\alpha(z)^{\top}}}
  \end{equation}
  of \(\alpha\) satisfies \(J_{\alpha}(\vec{z}) \equiv \sqrt{(n-1)!}\).  We now proceed to
  establish this.  To begin, we note that the derivative $D\alpha(z)$ can be written
  as an $(n-1) \times n(n-1)/2$ matrix
  \begin{equation}
    \begin{aligned}[c]
      M
      &= \inp{M_{i, (j,k)}}_{
        \substack{
        1 \leq i \leq n - 1\\
      1 \leq j \leq n - 1\\
      1 \leq k \leq j
      }}\text{, whose entries are given by}\\
      M_{i, (j, k)}
      &=
        \begin{cases}
          1 & \text{ if } j = i,\\
          -1 & \text{ if } j = i - 1, \text{ and }\\
          0 & \text{otherwise}.
        \end{cases}
    \end{aligned}
  \end{equation}
  Thus, we get
  \begin{equation}
    \label{eq:37}
    J_{\alpha}(\vec{z}) = \sqrt{\det\inp{D\alpha(z)D\alpha(z)^{\top}}} = \sqrt{\det\inp{MM^{\top}}}
  \end{equation}
  where $MM^{\top}$ is an $(n-1)\times (n-1)$ matrix whose entries are given by
  \begin{equation}
    (MM^{\top})_{i,j} =
    \begin{cases}
      2i-1 & \text{ if } i = j,\\
      -\min\inb{i,j} & \text{ if } \abs{i-j} = 1\text{, and}\\
      0 & \text{otherwise}.
    \end{cases}
  \end{equation}
  Consider now the following sequence of $(n-2)$ determinant preserving row
  operations on $MM^{\top}$: in the $i$th operation, row $i$ is added to row
  $i + 1$.  An induction shows that after the first $k$ operations,
  \begin{enumerate}
  \item all but the first $(k+1)$ rows of $MM^{\top}$ remain unchanged, and all
    entries with indices $(i,j)$ where $j > i$ remain unchanged; and
  \item the $(k+1)$th leading principal submatrix of $MM^{\top}$ is transformed
    into an upper triangular matrix with the entries $(1, 2, \dots, k+1)$ on the
    diagonal.
  \end{enumerate}
  From this, it follows that $\det\inp{MM^{\top}} = (n-1)!$.  Substituting this in
  \cref{eq:37} above, we get $J_{\alpha}(\vec{z}) = \sqrt{(n-1)!}$ for every
  $\vec{z} \in \tpGT(\lambda)$.  As noted above, this completes the proof of
  \cref{eq:24}.
\end{proof}

The following corollary is not required for further developments in this paper,
but we include it as an easy consequence of the identity in
\cref{lem:geom-quant}.  A stronger version of this corollary was proved by
Sra~\cite{sraInequalitiesNormalizedSchur2016}, and has been further generalized
by McSwiggen and Novak~\cite{mcswiggenMajorizationSphericalFunctions2022}.
\begin{corollary}\label{cor-s-monotone}
  Let $\mu$ be a non-increasing $n$-tuple of real numbers, and let
  $\delta_1 \geq \delta_2 \geq \dots \geq \delta_{n-1} \geq \delta_n$ be a
  non-increasing $n$-tuple of real numbers summing up to $0$ such that
  $\delta \succeq \vec{0}$ (i.e., the sum of every prefix is non-negative).  Let
  $\lambda \defeq \mu + \delta$ so that, in particular, $\lambda \succeq \mu$.
  Then, for any $x \in \R^n$, $S_{\lambda}(x) \geq S_{\mu}(x)$.
\end{corollary}
\begin{proof}
  Since $\lambda \succeq \mu$, $\mu \succeq \nu$ implies $\lambda \succeq
  \nu$. Thus, $\SH(\lambda) \supseteq \SH(\mu)$.  Further, $\GT(\delta, \vec{0})$
  is non-empty (since $\delta \succeq \vec{0}$), and for any
  $\nu \in \SH(\mu)$,
  $\GT(\lambda, \nu) \supseteq \GT(\mu, \nu) + \GT(\delta, \vec{0})$ (where the
  sum is a Minkowski sum), so that
  $V_{\lambda \nu} = \lvol{\GT(\lambda, \nu)} \geq \lvol{\GT(\mu, \nu)} = V_{\mu
    \nu}$.  Substituting these in \cref{eq:geom-quant} yields the claimed
  inequality.
\end{proof}

We will need the following modified form of the function $g_{\lambda, \mu}$
defined in \cref{prop:log-convex-cont-schur}.
\begin{observation}\label{ob:project}
  Define the map $q: \R^n \rightarrow \R^{n-1}$ by
  \begin{equation}
    q(x_1, x_2, \dots, x_n) \defeq (x_1-x_n, x_2 - x_n, \dots, x_{n-1} - x_n).
  \end{equation}
  Further, with $\lambda, \mu$ as in \cref{prop:log-convex-cont-schur} define
  $\tg: \R^{n-1} \rightarrow \R$ as
  \begin{equation}
    \tg_{\lambda, \mu}(y) \defeq
    \log
    \int\limits_{\tilde{\nu} \in \pSH(\lambda)}
    V_{\lambda, p_{\lambda}^{-1}(\tilde{\nu})}
    \cdot
    \exp(y\cdot(\tilde{\nu} - p(\mu)))
    \quad
    \prod\limits_{1 \leq i \leq n - 1} d\tilde{\nu}_{i}.
  \end{equation}
  Here $p: \SH(\lambda) \rightarrow \pSH(\lambda)$ and
  $p_{\lambda}^{-1}: \pSH(\lambda) \rightarrow \SH(\lambda) $ are as defined in
  \cref{eq:22,eq:63}.  Then we have, for all $x \in \R^n$, and all
  $\mu \in \SH(\lambda)$,
  \begin{equation}
    \label{eq-g-hat-to-g}
    \tg_{\lambda, \mu}(q(x)) = g_{\lambda,\mu}(x).
  \end{equation}
  In particular,
  \begin{equation}
    \label{eq:62}
    \mathop{\inf}\limits_{x \in \R^n} g_{\lambda,\mu}(x) =
    \mathop{\inf}\limits_{y \in \R^{n-1}} \tg_{\lambda,\mu}(y),
  \end{equation}
  and $\tg_{\lambda, \mu}$ is a convex function on $\R^{n-1}$.
\end{observation}
\begin{proof}
  \Cref{eq-g-hat-to-g} follows by substituting the following identity in the definition
  of $g_{\lambda,\mu}(x)$ (\cref{eq:38} combined with \cref{eq:geom-quant}): for
  any $x \in \R^n$ and $\nu, \mu \in \SH(\lambda)$ (which therefore satisfy
  $\abs{\mu} = \abs{\nu} = \abs{\lambda}$, so that
  $\nu_n = \abs{\lambda} - \abs{p(\nu)}$ and
  $\mu_n = \abs{\lambda} - \abs{p(\mu)}$),
  \begin{equation}
    x\cdot(\nu - \mu) = p(x)\cdot(p(\nu) - p(\mu)) + x_{n}(\abs{p(\mu)} - \abs{p(\nu)})
    = q(x)\cdot(p(\nu) - p(\mu)).
  \end{equation}
  \Cref{eq:62} then follows by noting that for every $y \in \R^{n-1}$,
  $g_{\lambda, \mu}((y_1, \dots,y_{n-1}, 0)) = \tg(y)$.  The convexity of
  $\tg_{\lambda,\mu}$ also follows from this representation of
  $\hat{g}_{\lambda, \mu}$ as a composition of the convex function
  $g_{\lambda,\mu}$ (\cref{prop:log-convex-cont-schur}) with a linear function.
\end{proof}
We can now prove the requisite lower bound on the volume of the Kostka polytope
\(\GT(\lambda, \mu)\).
\begin{proposition}\label{prop:correct-mean}
  Let $\lambda$, $\mu$ and $g_{\lambda, \mu}$ be as in the statement of
  \cref{prop:log-convex-cont-schur}.  Consider
  \begin{displaymath}
    x^{*} \in \mathop{\arg\min}\limits_{x \in \R^n} g_{\lambda,\mu}(x).
  \end{displaymath}
  Then, the following are true.
  \begin{enumerate}
  \item The mean of the probability distribution with density proportional to
    $f_{\lambda,x^{*}}$ (as defined in \cref{eq:f-def}) is
    $\mu$. \label{item:correct-mean}
  \item \label{item:lower-bound} $V_{\lambda, \mu} \geq
\frac{\sqrt{(n-1)!} \cdot S_{\lambda}(x^{*}) }{\exp(n + x^{*}\cdot\mu) \cdot \lvol{\pSH(\lambda)} }$. \end{enumerate}
\end{proposition}

\begin{proof}
  From \cref{prop:log-convex-cont-schur}, $g_{\lambda,\mu}$ is a differentiable
  convex function on $\R^{n}$.  Further, using the notation of
  \cref{them:fradelizi,cor:log-concave-f}, the gradient $Dg_{\lambda, \mu}$ of
  $g_{\lambda,\mu}$ is given by
  \begin{equation}
    \label{eq:15}
    Dg_{\lambda, \mu}(w) = \mean{f_{\lambda, w}} - \mu.
  \end{equation}
  Since $g_{\lambda, \mu}$ is convex and differentiable, the gradient must
  vanish at any minimum of $g_{\lambda, \mu}$.  \Cref{item:correct-mean}
  therefore would follow from \cref{eq:15}, if we establish the technical point
  that the minimum is actually achieved at some finite $x^{*}$.  This follows
  from \cref{eq-g-hat-to-g} which relates values of $\tg_{\lambda, \mu}$ to
  those of $g_{\lambda, \mu}$, and from the lower bound proportional to
  $\norm{y}$ on $\tg_{\lambda, \mu}(y)$ (for large $\norm{y}$) shown in
  \cref{lem:g-lower-bound} below, which implies that $\tg_{\lambda, \mu}$ is a
  convex function which must achieve its minimum in a compact subset of
  $\R^{n-1}$.

  To obtain \cref{item:lower-bound}, we now pick an $x^{*}$ as in
  \cref{item:correct-mean}, and consider the function
  $f_{\lambda, x^{*}}(\nu) = V_{\lambda \nu}\exp(x^{*}\cdot \nu)$ defined in \cref{eq:f-def}.  By
  \cref{cor:log-concave-f}, $f_{\lambda, x^{*}}$ is a log-concave function supported
  on the closed convex set \(\SH(\lambda)\).  Further, after a linear transformation,
  the convex set \(\SH(\lambda)\) can be viewed as a convex body in \(\R^{n'}\) for
  some \(n' < n\).  Thus, we can apply \cref{them:fradelizi} to upper bound the
  maximum value of \(f_{\lambda, x^{*}}\).  Applying \cref{them:fradelizi} to
  $f_{\lambda, x^{*}}$, and recalling from \cref{item:correct-mean} that
  $\mean{f_{\lambda,x^{*}}} = \mu$, we therefore get the inequality
  \begin{equation}
    f_{\lambda, x^{*}}(\nu) \leq \exp(n') \cdot f_{\lambda, x^{*}}(\mu) \leq \exp(n) \cdot f_{\lambda, x^{*}}(\mu).
  \end{equation}
  for every $\nu \in \SH(\lambda)$.  Substituting this on the right hand side of
  \cref{eq:geom-quant} (and using the definition of $f_{\lambda, x^{*}}$ in
  \cref{eq:f-def}), we thus get
  \begin{equation}
    \sqrt{(n-1)!}\cdot
    S_{\lambda}(x^{*})
\leq
\exp(n)
\cdot V_{\lambda, \mu}
\cdot \exp(x^{*}\cdot \mu)
\cdot \lvol{{\pSH(\lambda)}}.
\end{equation}
  Rearranging, this gives \cref{item:lower-bound}.
\end{proof}

\section{Upper bounding the volume: Geometric properties of
     \texorpdfstring{$\GT(\lambda, \mu)$}{GT(λ, μ)}}
\label{sec-geom-pro}
In this section, we prove an upper bound on the volume of the Kostka polytope
\(\GT(\lambda, \mu)\), by analyzing the geometry of this polytope.  The main result of
this section is \cref{prop:vol-upper-estimate}, which gives such an upper bound.
Towards this, we first develop an understanding of how the volume of
\(\GT(\lambda, \mu)\) changes as \(\mu\) is perturbed (see \cref{lem:2.12}).  This, in
turn, is achieved by showing that there exists a point in the Kostka polytope
that is sufficiently far from all bounding hyperplanes (see
\cref{lem:2.10,prop-pgt-ball}), and then studying homotheties of the polytope
centered at this point (see the proof of \cref{lem:2.12}).

The above strategy works well when the pair \((\lambda, \mu)\) is ``well-conditioned''
in a certain sense: we now proceed to define certain quantities that formalize
these conditioning properties.

\begin{definition}[\textbf{Conditioning properties of $(\lambda, \mu)$}]\label{def-gap}
  First, given a vector $v \in \R^n$, let
  $\overline{v} \defeq \frac{\sum_{i=1}^n v_i}{n}$. Let $\one \in \R^n$ be the
  vector, every coordinate of which equals $1.$

  Given a non-increasing $n$-tuple $\lambda$, we define
  $\gap \defeq \min\limits_{i \in [n-1]}(\la_i - \la_{i+1}),$ so that
  $\lvol{\GT(\lambda)} = 0$ when ${\gap} = 0$.  A related quantity is
  \begin{equation}\label{eq-r0}
    r_0 = r_0(\lambda) \defeq
    \max\inb{\delta \st \abs{\lambda} = \abs{x} \text{ and }
      \norm[2]{\bar{\lambda}\vec{1} - x
      } \leq \delta \implies x \in \SH(\lambda)}.
  \end{equation}
  One can show that $r_0(\lambda) \geq \gap \sqrt{n-1}/2$: see \cref{lem-r0-bound} below.
We now turn to $\mu$.  We define
 \begin{equation}
   \begin{aligned}
     \tau = \tau(\la, \mu)
     &\defeq \sup\{\de \in [0, 1]
       \st
       \frac{\mu - \overline{\la}\one}{1-\de} + \overline{\la}\one \in \SH(\la)\}\\
     &= \sup\{\de \in [0, 1]
       \st \exists \nu \in \SH(\lambda) \text{ for which }
       \mu = \delta\overline{\lambda}\one + (1-\delta)\nu
       \},
   \end{aligned}   \label{eq:tau-def}
 \end{equation}
 which is a measure of how far $\mu$ is from the boundary of $\SH(\lambda)$.  In
 particular, if $\lambda$ and $\mu$ have integer entries and $\mu$ is in the
 relative interior of $\SH(\lambda)$, then
 $\tau(\lambda,\mu) \geq \frac{1}{n\lambda_1}$ (see \cref{lem-tau-lower-bound}).

 In
 \cref{lem:2.10} we study the related quantity
 \begin{equation}
   \label{eq-r-def}
   r(\lambda, \mu) \defeq \tau(\la, \mu) \gap/4.
 \end{equation}
 Given the above quantities, we define below a condition parameter
 $\bdrad(\lambda, \mu)$ for $\GT(\lambda, \mu)$, which turns out to be important
 in our analysis. Roughly speaking, we view $\GT(\lambda, \mu)$ as being well
 conditioned when $\bdrad(\lambda, \mu)$ is large.
 \begin{equation}
    \label{eq-bdrad-def}
    \bdrad(\lambda, \mu) \defeq\frac{1}{\sqrt{n}}
    \cdot \min\inb{
      \tau(\lambda, \mu)r_0(\lambda),
      \frac{r(\lambda, \mu)}{4n^{3/2}}
    } = \frac{\tau(\lambda, \mu)}{\sqrt{n}}
    \cdot \min\inb{
      r_0(\lambda),
      \frac{\gap}{16 n^{3/2}}
    }.
  \end{equation}
\end{definition}
\paragraph{\textbf{Positivity of the conditioning parameter}} Using the lower
bound on \(r_{0}(\lambda)\) in \cref{lem-r0-bound} below, we immediately obtain
the following lower bound on \(\bdrad(\lambda, \mu)\):
\begin{equation}
  \label{eq:1}
  \bdrad(\lambda, \mu) \geq \frac{\tau(\lambda, \mu)\cdot \gap}{16 n^2}.
\end{equation}
Thus, we see that when \(\lambda\) has distinct entries (thus, \({\gap} > 0\))
and \(\mu\) lies in the relative interior of \(\SH(\lambda)\) (so that
\(\tau(\lambda, \mu) > 0\)), \(\bdrad(\lambda, \mu)\) is strictly positive.
Note that when either of these two conditions fails, the polytope
\(\GT(\lambda, \mu)\) lies in an affine subspace of dimension strictly less than
\(\binom{n-1}{2}\), and hence its \(\binom{n-1}{2}\)-dimensional volume
\(V_{\lambda, \mu}\) is \(0\).

\begin{lemma}[\textbf{Lower bound on $r_{0}(\lambda)$}]
  \label{lem-r0-bound}
  Let $\lambda$ be a non-increasing $n$-tuple.  Then
  $r_{0}(\lambda) \geq \gap\sqrt{n-1}/2$.
\end{lemma}
\begin{proof}
  The point $\overline{\lambda}\one$ is the average of all permutations of
  $\lambda$ and thus lies in $\SH(\lambda)$.  We now compute its distance from
  the hyperplanes corresponding to the inequality constraints in the definition
  of $\SH(\lambda)$ (\cref{eq:majorization}).  Consider the $k$th inequality
  constraint in \cref{eq:majorization} (for any permutation $\tau$, which is not
  really relevant since all entries of $\overline{\lambda}\one$ are the same).
  A direct computation then gives
  \begin{equation}
    \label{eq:2}
    \begin{gathered}
      \sum_{j=k+1}^n\lambda_j \leq (n-k)\lambda_k - \gap\sum_{j=k+1}^{n}(j-k) =
      (n-k)\lambda_{k} - \frac{(n-k)(n-k+1)\gap}{2}\text{, and }\\
      \sum_{j=1}^k\lambda_j \geq k\lambda_k + \gap\sum_{j=1}^k{(k - j)} =
      k\lambda_k + \frac{k(k-1)\gap}{2}.
    \end{gathered}
  \end{equation}
  Since $\overline{\lambda} = \frac{1}{n}\sum_{j=1}^n\lambda_j$, the slack in the
  $k$th inequality in \cref{eq:majorization} is therefore
  \begin{equation}\label{eq-mean-diff}
    \begin{aligned}
      \sum_{j=1}^k\lambda_{j} - k\overline{\lambda}
      = \inp{1 -\frac{k}{n}}\sum_{j=1}^k\lambda_j -
      \frac{k}{n}\sum_{j=k+1}^n\lambda_j
      \overset{\textup{\cref{eq:2}}}{\geq}
\frac{k(n-k)\gap}{2}.
    \end{aligned}
  \end{equation}
  Since the coefficient vector of this hyperplane has norm $\sqrt{k}$, it
  follows that the distance of $\overline{\lambda}\one$ from hyperplane(s)
  corresponding to the $k$th inequality in \cref{eq:majorization} is at least
  $\sqrt{k}(n-k)\gap/2$.  Taking the minimum of these distances over all
  $1 \leq k \leq n -1$, we thus get $r_{0}(\lambda) \geq \gap\sqrt{n-1}/2$.
\end{proof}

\begin{lemma}[\textbf{Lower bound on
    $\tau(\lambda, \mu)$}]\label{lem-tau-lower-bound}
  Let $\lambda$ be a partition with $n$ parts, and for $\mu \in \SH(\lambda)$, let
  $d_{\lambda\mu}$ be the shortest distance of $\mu$ from any of the hyperplanes
  corresponding to the inequality constraints defining $\SH(\lambda)$
  (\cref{eq:majorization}).  Then
  $\tau(\lambda, \mu) \geq \frac{d_{\lambda \mu}}{\lambda_1\sqrt{n}}$.  In particular, if
  $\lambda$ and $\mu$ have integer entries and $\mu$ is in the relative interior of
  $\SH(\lambda)$, then $d_{\lambda \mu} \geq \frac{1}{\sqrt{n}}$, so that
  $\tau(\lambda,\mu) \geq \frac{1}{n\lambda_1}$.
\end{lemma}
\begin{proof}
  Consider the line segment joining $\overline \la$ and $\mu$, and extend it to
  the point $ \frac{\mu - \overline{\la}\one}{1-\tau} + \overline{\la}\one$ on
  the boundary of $\SH(\lambda)$, which lies on one of the hyperplanes
  (corresponding to an inequality constraint in \cref{eq:majorization}) bounding
  $\SH(\la)$, which we denote by $H$.  Given a point $x$, let $\dist(x, H)$
  denote the Euclidean distance of the point $x$ to $H$, so that
  $\dist(\mu, H) \geq d_{\lambda \mu}$.  Now, by similar triangles,
  $$\frac{ \dist(\mu, H)}{\dist(\overline \la \vec{1}, H)} = \tau.$$
  The first claim now follows since a consideration of the slack in any of the
  inequality constraints in \cref{eq:majorization} for the point
  $\overline{\lambda}\one$ shows that
  $\dist(\overline{\lambda}\one, H) \leq \sqrt{n}\lambda_1$.

  If $\mu$ and $\la$ are integral, $\mu$ is in the relative interior of
  $\SH(\la)$, and $H$ is a hyperplane of the form given by
  (\ref{eq:majorization}), then the slack of the corresponding inequality
  constraint is at least $1$, so that the distance between $\mu$ and $H$ is at
  least $\frac{1}{\sqrt{n}}$.  Thus, in this case $d_{\lambda\mu} \geq 1/\sqrt{n}$, as
  claimed.
\end{proof}

Substituting the above bounds in \cref{eq-bdrad-def} thus gives that when
\(\lambda\) is a partition with \(n\) parts that are distinct integers, and
\(\mu\) is an integral point in the relative interior of \(\SH(\lambda)\), then
\begin{equation}
  \label{eq-bdrad-bound}
  \bdrad(\lambda, \mu) \geq \frac{\gap}{16 \lambda_1  n^3} \geq \frac{1}{16 \lambda_1  n^3}.
\end{equation}

\begin{figure}
  \centering
  \includegraphics{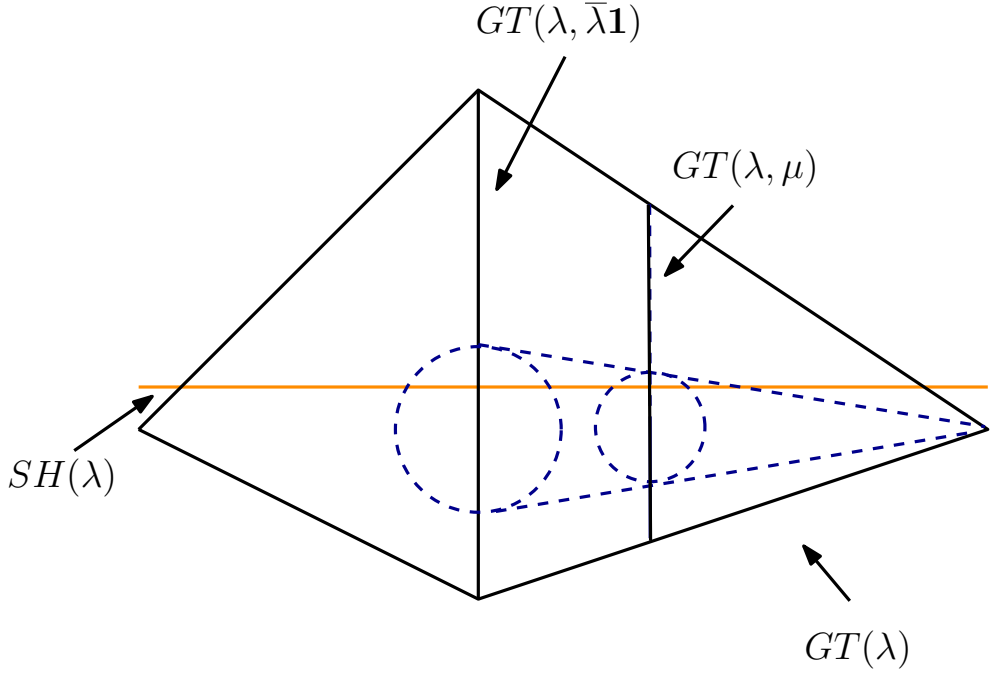}
  \caption{Convexity leads to a lower bound on the diameter of the largest ball contained in $\GT(\la, \mu)$.}
  \label{fig:convexity}
\end{figure}

\begin{lemma}\label{lem:2.10}
  Let \(\lambda\) be a non-increasing \(n\)-tuple with \(\gap > 0\), and assume that
  \(\mu \in \SH(\lambda)\). Then, $\GT(\la, \mu)$ contains an
  ${\binom{n-1 }{2}}$-dimensional Euclidean ball $B$ of radius
  $r(\lambda, \mu) = \tau(\la, \mu) \gap/4.$ Moreover, $\GT(\la)$ contains an
  ${\binom{n }{2}}$-dimensional ball $B'$ of the same radius and center
  $x_{\la\mu}$ as $B$.
\end{lemma}
\begin{proof}
  We begin by proving the following claim.
  \begin{claim}
    $\GT(\la, \overline{\la}\one)$ contains an ${\binom{n-1 }{2}}$-dimensional
    Euclidean ball $B_1$ of radius \(\gap/4\). Moreover, $\GT(\la)$ contains an ${\binom{n }{2}}$-dimensional ball $B_1'$
    of the same radius and center as $B_1$.

\end{claim}
\begin{proof}[Proof of claim]
  Let $\nu = ((n-1)\gap, \dots, {\gap}, 0).$ In other words,
  $\nu_i = (n-i){\gap}$ for $1 \leq i \leq n$. Let $\gamma := \la - \nu$. Note that
  $\gamma$ is a non-increasing \(n\)-tuple with \(\gap[\gamma] = 0\).  We begin by proving
  that $\GT(\nu, \overline{\nu}\one)$ contains a point $x_\nu$ centered at which there
  is an ${\binom{n }{2}}$-dimensional ball of radius \(\gap/4\) that is
  contained in $\GT(\nu)$.  Consider the point
  $x_\nu \in \GT(\nu, \overline{\nu}\one)$ given by
  \begin{equation}
    x_{i\, j} = \inp{\frac{n+i}{2} - j}\gap \text { for \(1 \leq i \leq n \) and \(1 \leq
      j \leq i\)}.
  \end{equation}
Examining the slacks in the inequalities in \eqref{eq:gt-cons} corresponding to the Gelfand-Tsetlin polytope, we see that
  this point is the center of an ${\binom{n }{2}}$-dimensional ball of radius
  $\frac{\gap}{2\sqrt{2}} \geq {\gap}/4$ contained in $\GT(\nu)$.

  Next, since $\gamma \succeq \overline{\gamma}\one$, we have that
  $\GT(\gamma, \overline{\gamma}\one)$ is a nonempty polytope.  Lastly, we observe that
  $\GT(\la, \overline{\la}\one) \supset \GT(\nu, \overline{\nu}\one) + \GT(\gamma,
  \overline{\gamma}\one), $ and $\GT(\la) \supset \GT(\nu) + \GT(\gamma).$ As a consequence, we
  have the claim.
\end{proof}
Using a similar triangles argument, it now follows from the definition of
$\tau(\lambda, \mu)$ and the convexity of $\GT(\la)$ (see \cref{fig:convexity})
that $\GT(\la)$ contains a ball of radius $\tau(\la, \mu) {\gap}/4$ centered at
a point in $\GT(\la, \mu)$ and the lemma is proved.
\end{proof}

In preparation for the following lemma, we need the following simple proposition.
\begin{proposition}
  Let $x \in \GT(\lambda, \mu)$ be such that an $\binom{n-1}{2}$-dimensional
  Euclidean ball of radius $\delta > 0$ centered at \(x\) is contained in
  $\GT(\lambda, \mu)$.  Then, the $\binom{n-1}{2}$-dimensional Euclidean ball of
  radius $\delta/\sqrt{n-1}$ centered around $\tp(x)$ is contained in
  $\pGT(\lambda, \mu)$.\label{prop-pgt-ball}
\end{proposition}
\begin{proof}
  Consider a point $\tilde{z}$ such that
  $\norm{\tilde{z} - \tp(x)} \leq \delta/\sqrt{n-1}$.  Consider the point $z$
  defined as follows:
  \begin{equation}
    \begin{gathered}
      z_{i\,j} = \tilde{z}_{i\,j}
      \text { for }2 \leq i \leq n -1, 1 \leq j \leq i - 1,\\
      z_{n\, j} = \lambda_{j} \text{ for } 1 \leq j \leq n\text{, and}\\
      z_{i\,i} = \sum_{k=1}^i\mu_k - \sum_{j=1}^{i-1}\tilde{z}_{i\,j}\text{ for }
      1 \leq i \leq n - 1.
    \end{gathered}
  \end{equation}
  By construction, $\tp(z) = \tilde{z}$, so to establish that
  $\tilde{z} \in \pGT(\lambda, \mu)$, all we need to show is that
  $z \in \GT(\lambda, \mu)$.  Now, since all the equality constraints defining
  $\GT(\lambda, \mu)$ are satisfied by $z$ by construction, we only need to show
  that $\norm{z - x} \leq \delta$ to show that $z \in \GT(\lambda, \mu)$.  This
  follows by the following direct computation:
  \begin{equation}
    \begin{aligned}
      \norm{z-x}^2 - \norm{\tilde{z} - \tp(x)}^2 =
      \sum_{i = 1}^{n-1}\inp{\sum_{j=1}^{i-1}(\tilde{z}_{i\,j} - \tp(x)_{i\,j})}^2
      &\leq (n-2)\sum_{i = 1}^{n-1}\sum_{j=1}^{i-1}(\tilde{z}_{i\,j} -
        \tp(x)_{i\,j})^2\\
      &\leq (n-2)\norm{\tilde{z} - \tp(x)}^2.
    \end{aligned}\qedhere
  \end{equation}
\end{proof}

\begin{lemma}\label{lem:2.12}
  Let $r = r(\lambda, \mu) = \tau(\la, \mu){\gap}/4$ be as defined in
  \cref{lem:2.10}.  Then,
  \begin{equation}
    \label{eq:53}
    V_{\lambda, \mu}
    \geq
    \sqrt{(n-1)!} \cdot
    \inp{\frac{r}{\sqrt{n}}}^{\binom{n-1}{2}}\cdot b_{\binom{n-1}{2}},
  \end{equation}
  where $b_{k} = \frac{\pi^{\frac{k}{2}}}{\Gamma(\frac{k}{2} + 1)}$ denotes the
  volume of the Euclidean ball of unit radius in $\R^{k}$.  If
  $\mu' \in \SH(\lambda)$ is such that
  $\|\mu - \mu'\|_2 < \de < \frac{r}{2\sqrt{n-1}}$, then for $n \geq 3$
  \begin{equation}
    \label{eq:44}
    V_{\lambda, \mu' }
    \geq
    V_{\lambda,\mu}
    \left(1 - \frac{2\delta\sqrt{n-1}}{r}\right)^{\binom{n-1}{2}}.
  \end{equation}
\end{lemma}
\begin{proof}
  \cref{lem:2.10} combined with \cref{prop-pgt-ball} shows that $\pGT(\la, \mu)$
  contains an ${\binom{n-1 }{2}}$ dimensional ball of radius $r/\sqrt{n-1}$
  centered at $\tp(x_{\la\,\mu})$ (where $x_{\lambda\,\mu}$ is as in
  \cref{lem:2.10}).  This gives (recalling the notation $\tV_{\lambda, \mu}$ for
  $\lvol{\pGT(\lambda, \mu)}$)
  \begin{equation}
    V_{\lambda, \mu}
    \overset{\text{\cref{eq:vol-relation}}}{=}
    \sqrt{(n-1)!} \cdot \tV_{\lambda, \mu}
    \geq
    \sqrt{(n-1)!} \cdot
    \inp{\frac{r}{\sqrt{n-1}}}^{\binom{n-1}{2}}\cdot b_{\binom{n-1}{2}},
  \end{equation}
  which proves \cref{eq:53}.  We now turn to the proof of \cref{eq:44}.  We saw
  above that $\pGT(\lambda,\mu)$ contains an Euclidean ball of radius
  $r/\sqrt{n-1}$ centered at $\tp(x_{\lambda\,\mu})$.  \Cref{prop-hyperplane-move}
  implies that each bounding hyperplane of $\pGT(\lambda, \mu')$ is obtained by
  translating a corresponding bounding hyperplane of $\pGT(\lambda, \mu)$ by a distance
  of at most $\frac{\sqrt{n-1}}{\sqrt{n-2}}{\delta} \leq 2\delta$ for
  $n \geq 3$. Since $\delta < \frac{r}{2\sqrt{n-1}}$, it thus follows that the image of
  $\pGT(\lambda, \mu)$ under a homothety by a factor of
  $1 - \frac{2\delta\sqrt{n-1}}{r}$ centered at $\tp(x_{\lambda, \mu})$ is contained in
  $\pGT(\lambda, \mu')$.  This shows that
  \begin{equation}
    \tV_{\lambda, \mu' }
    \geq \tV_{\lambda,\mu}\left(1 - \frac{2\delta\sqrt{n-1}}{r}\right)^{\binom{n- 1}{2}}.
  \end{equation}
  \Cref{eq:44} now follows from the scaling between $V_{\lambda,\mu}$ and
  $\tV_{\lambda, \mu}$ given in \cref{eq:vol-relation}.
\end{proof}
For future use, we record the following computation. Let
\begin{equation}
  \label{eq-delta-prime-def}
  \de' = \delta'(\lambda, \mu) := \frac{r(\lambda,\mu)}{4n^{3/2}} = \frac{\tau(\lambda,\mu)\gap}{16 n^{3/2}}.
\end{equation}
By \cref{eq:44} in Lemma~\ref{lem:2.12}, if $\|\mu - \mu'\|_2 < \de'$, then for $n \geq 3$
(using the inequality \(1 - x \geq \exp(-2x)\) valid for \(0 \leq x \leq 1/2\))
\begin{equation}
  \label{eq:vol-perturbed-upper}
  V_{\la, \mu'} \geq
    \left(1 - \frac{2\delta'\sqrt{n-1}}{r}\right)^{\binom{n-1}{2}}
    \cdot V_{\lambda,\mu}  \geq \exp\inp{\frac{-2\delta'n^{5/2}}{r}}\cdot V_{\lambda,\mu}
  = e^{-n/2} V_{\lambda, \mu}.
\end{equation}
Let $\tau(\lambda, \mu)$ be as defined in \cref{eq:tau-def}.  Recall that (from
\cref{eq-r0})
\begin{equation}
  r_0 = r_0(\lambda) \defeq
  \max\inb{\delta \st \abs{\lambda} = \abs{x} \text{ and }
    \norm[2]{\bar{\lambda}\vec{1} - x
    } \leq \delta \implies x \in \SH(\lambda)}.
\end{equation}

\begin{claim}
  \label{clm:shur-ball}
  With $\lambda$ and $\mu$ as above, $\pSH(\lambda)$ contains an Euclidean ball
  $B$ of radius $r_{0}(\lambda)\tau(\lambda, \mu)/\sqrt{n}$ centered at
  $p(\mu)$.  Further, if $\mu'$ is such that $p(\mu') \in B$ then
  \begin{equation}
    \label{eq:47}
    \norm[2]{\mu - \mu'} \leq \norm[2]{p(\mu) -p(\mu')} \cdot \sqrt{n} \leq r_{0}(\lambda)\tau(\lambda, \mu).
  \end{equation}
\end{claim}
\begin{proof}
  Consider the Euclidean ball $B'$ of radius $r_0(\lambda)/\sqrt{n}$ centered at
  $p(\bar{\lambda}\vec{1}) \in \pSH(\lambda)$.  Consider a point
  $\eta' = (\eta'_1, \eta_2', \dots, \eta_{n-1}') \in B'$ and define
  $\eta \defeq (\eta'_1, \eta_2', \dots, \eta_{n-1}' , \abs{\lambda} -
  \abs{\eta'})$.  We then have $\abs{\eta} = \abs{\lambda}$, and
  \begin{equation}
    \label{eq:40}
    \begin{aligned}[c]
      \norm[2]{\eta - \bar{\lambda}\vec{1}}^2
      &= \norm[2]{\eta' - p(\bar{\lambda}\vec{1})}^2
        + \inp{\vec{1}^T(\eta' - p(\bar{\lambda}\vec{1}))}^2\\
      &\leq \frac{r_{0}(\lambda)^2}{n} + (n-1)\frac{r_{0}(\lambda)^2}{n} \leq r_0(\lambda)^2,
    \end{aligned}
  \end{equation}
  so that, by the definition of $r_0(\lambda)$, we have $\eta \in \SH(\lambda)$, which implies
  that $p(\eta) = \eta' \in \pSH(\lambda)$.  Thus, the Euclidean ball $B'$ of radius
  $r_0(\lambda)/\sqrt{n}$ centered at $p(\bar{\lambda}\vec{1})$ lies in $\pSH(\lambda)$.

  By the definition of $\tau(\lambda, \mu)$ in \cref{eq:tau-def}, there exists
  $\nu \in \SH(\lambda)$ such that $\mu$ is a convex combination of $\nu$ and
  $\bar{\lambda}\vec{1}$, given by
  $\mu = \tau(\lambda, \mu)\bar{\lambda}\vec{1} + (1-\tau(\lambda, \mu))\nu$.
  This also implies that $p(\mu) = \tau(\lambda, \mu)p(\bar{\lambda}\vec{1}) + (1-\tau(\lambda,
  \mu))p(\nu)$, with $p(\bar{\lambda}\vec{1}), p(\nu) \in \pSH(\lambda)$.  The
  claim now follows because of the convexity of $\pSH(\lambda)$, and since the
  ball $B'$ above lies in $\pSH(\lambda)$.  More explicitly, any point
  $\tilde{\mu}$ satisfying
  $\norm[2]{\tilde{\mu} - p(\mu)} \leq r_{0}(\lambda)\tau(\lambda, \mu)/\sqrt{n}$ can
  be written as the convex combination
  \begin{equation}
    \label{eq:39}
    \tilde\mu = \tau(\lambda, \mu)\tilde\lambda + (1-\tau(\lambda, \mu))p(\nu),
  \end{equation}
  where $\lambda'$ satisfies
  $\norm[2]{\tilde\mu - p(\mu)} = \tau(\lambda, \mu)\norm[2]{\tilde\lambda - p(\bar{\lambda}\vec{1})}$, so
  that $\tilde\lambda \in B' \subseteq \pSH(\lambda)$.  That \(\tilde\mu\) is in
  \(\pSH(\lambda)\) then follows from \cref{eq:39} since $\pSH(\lambda)$ is a convex set.
  This shows that the ball $B$ as defined in the claim is contained in
  $\pSH(\lambda)$.  Let $\mu'$ be such that $p(\mu') = \tilde\mu$.  The claimed upper bound
  on $\norm[2]{\mu' - \mu}$ in \cref{eq:47} then follows from the definition of
  $B$ and a calculation formally identical to the one in \cref{eq:40}.
\end{proof}

\begin{proposition}\label{prop:vol-upper-estimate}
  Let $\tau(\lambda, \mu)$ and $r_0(\lambda)$ be as defined in
  \cref{eq:tau-def,clm:shur-ball} respectively. Let
  \begin{equation}
    \bdrad(\lambda, \mu) = \frac{\tau(\lambda, \mu)}{\sqrt{n}}
    \cdot \min\inb{
      {r_0(\lambda)},
      \frac{\gap}{16n^{3/2}}
    },
  \end{equation}
  be as defined in \cref{eq-bdrad-def} and let $b_{n-1}$ denote the volume of
  the Euclidean ball of unit radius in $\R^{n-1}$.  Then, we have
  \begin{equation}
V_{\lambda,\mu} \leq \frac{\sqrt{(n-1)!}
      \cdot
      \exp\inp{
        n/2
      }
      \cdot
      S_{\lambda}(x)
    }{
      \exp(x\cdot\mu)\cdot b_{n-1}\cdot \bdrad(\lambda, \mu)^{n-1}
    }
  \end{equation}
  for all $x \in \R^n$.
\end{proposition}
\begin{remark}
  Recall from the discussion following \cref{eq-bdrad-def} that
  \(\bdrad(\lambda, \mu)\) is positive whenever \(\lambda\) and \(\mu\) are such
  that the volume \(V_{\lambda, \mu}\) is non-zero.
\end{remark}
\begin{proof}
  Let $B'$ denote the Euclidean ball of radius $\bdrad(\lambda,\mu)$ in
  $\R^{n-1}$, centered at $p(\mu) \in \pSH(\lambda)$.  The definition of
  $\bdrad(\lambda,\mu)$ and \cref{clm:shur-ball} then imply that
  $B' \subseteq B \subseteq \pSH(\lambda)$ (where $B$ is as defined in
  \cref{clm:shur-ball}).  Again, the definition of $\bdrad(\lambda, \mu)$ and
  \cref{clm:shur-ball} (cf.~\cref{eq:47}) then imply that if $\nu$ is such that
  $p(\nu) \in B'$, then
  \begin{equation}
    \norm[2]{\mu - \nu} \leq \sqrt{n}\cdot \bdrad(\lambda, \mu) \leq
    \delta'(\lambda, \mu),
  \end{equation}
  where $\delta'$ is as defined in \cref{eq-delta-prime-def}.  \Cref{lem:2.10}
  (cf.~\cref{eq:vol-perturbed-upper}) then implies that for any $\nu$ for which
  $p(\nu) \in B'$,
  \begin{equation}
    \label{eq:49}
    V_{\lambda, \nu} \geq  e^{-n/2} V_{\lambda, \mu}.
  \end{equation}
  Substituting these estimates in \cref{eq:geom-quant}, we get
  \begin{align}
    \sqrt{(n-1)!}\cdot S_{\lambda}(x)
    &= \int\limits_{\nu \in \SH(\lambda)}
      V_{\la\nu} \exp(x \cdot \nu)
      \quad
      \prod\limits_{1 \leq i \leq n - 1} d\nu_{i}\\
    &\geq \int\limits_{\nu \in \SH(\lambda) \st p(\nu) \in B'}
      V_{\la\nu} \exp(x \cdot \nu)
      \quad
      \prod\limits_{1 \leq i \leq n - 1} d\nu_{i}\\
    &\overset{\textup{\cref{eq:49}}}{\geq}
      e^{-n/2} V_{\lambda, \mu}
      \cdot \int\limits_{\nu \in \SH(\lambda) \st p(\nu) \in B'}
      \exp(x \cdot \nu)
      \prod\limits_{1 \leq i \leq n - 1} d\nu_{i}.\label{eq:46}
  \end{align}
  Since $B'$ is a ball centered at $p(\mu)$, it is centrally symmetric around
  $p(\mu)$.  The change of variables
  $(\nu_1, \nu_{2}, \dots, \nu_{n-1}) \mapsto 2p(\mu) - (\nu_1, \nu_{2}, \dots,
  \nu_{n-1})$ thus gives
  \begin{align}
    \int\limits_{\nu \in \SH(\lambda) \st p(\nu) \in B'}
    \exp(x \cdot \nu)
    \prod\limits_{1 \leq i \leq n - 1} d\nu_{i}
    &= \int\limits_{\nu \in \SH(\lambda) \st p(\nu) \in B'}
      \exp(x \cdot (2\mu - \nu))
      \prod\limits_{1 \leq i \leq n - 1} d\nu_{i}\\
    &= \int\limits_{\nu \in \SH(\lambda) \st p(\nu) \in B'}
      \frac{1}{2}\insq{
      \exp(x \cdot \nu)
      + \exp(x \cdot (2\mu - \nu))
      }
      \prod\limits_{1 \leq i \leq n - 1} d\nu_{i}\\
    &\geq \exp(x\cdot \mu) \cdot \lvol{B'},
  \end{align}
  where the last inequality follows since the exponential is a convex function.
  Substituting this back into \cref{eq:46}, noting that
  $\lvol{B'} = \bdrad(\lambda, \mu)^{n-1}\cdot b_{n-1}$, and rearranging, we get
  the claim.
\end{proof}

\section{Main result and computational aspects}
\label{sec:main-result-comp}
We now combine the upper and lower and upper bounds on the volumes of Kostka
polytopes obtained in \cref{sec-cont-anal-schur,sec-geom-pro} to complete the
proof of our main result (\cref{thm:main}).  Then, in \cref{sec:comp-aspects},
we describe how to efficiently compute the estimate that \cref{thm:main} yields.

From the discussion following \cref{eq-bdrad-def}, it follows that if
\(\lambda\) does not have distinct parts or if \(\mu\) is not in the relative
interior of \(SH(\lambda)\) then \(\lvol{\GT(\lambda, \mu)} = 0\). We thus state
the result only for the case when \(\lambda\) has distinct parts and \(\mu\) is
in the relative interior of \(\SH(\lambda)\), and in this case
\(\bdrad(\lambda, \mu)\) is positive.  Thus, fix a partition \(\lambda\) with
\(n\) distinct parts, and a vector \(\mu\) in the relative interior of
\(SH(\lambda)\).  Recall from the statement of
Proposition~\ref{prop:correct-mean} the definition
\begin{displaymath}
    x^{*} \in \mathop{\arg\min}\limits_{x \in \R^n} g_{\lambda,\mu}(x).
  \end{displaymath}
Pick $\bdrad(\lambda, \mu)$ as in \cref{prop:vol-upper-estimate}, and consider the
estimate
\begin{equation}
  F(\lambda, \mu) \defeq
\frac{\sqrt{(n-1)!}
      \cdot
      \exp\inp{
        n/2
      }
      \cdot
      S_{\lambda}(x^{*})
    }{
      \exp(x^{*}\cdot\mu)\cdot b_{n-1}\cdot \bdrad(\lambda, \mu)^{n-1}
    }.
\end{equation}
\Cref{prop:correct-mean,prop:vol-upper-estimate} then show that
\begin{equation}
  1
  \leq
  \frac{F(\lambda, \mu)}{V_{\lambda, \mu}}
  \leq
  \frac{
    \exp\inp{
      3n/2
    }
    \cdot
    \lvol{\pSH(\lambda)}
  }{
    b_{n-1}\cdot \bdrad(\lambda, \mu)^{n-1}
  }.
\end{equation}
Denoting
$g^{*}_{\lambda, \mu} \defeq \mathop{\inf}\limits_{y \in
  \R^{n-1}}\tg_{\lambda,\mu}(y)$, \cref{eq:62} in \cref{ob:project} thus gives us the following:
\begin{equation}
  \label{eq:opt-approx}
  \frac{\sqrt{(n-1)!}  \exp(-n)}{\lvol{\pSH(\la)}}
  \leq
  \frac{V_{\la\mu}}{\exp\inp{g^{*}_{\lambda,\mu}}}
  \leq
  \frac{\sqrt{(n-1)!}
    \exp(n/2){\Gamma(\frac{n+1}{2})}}{{\pi^{\frac{n-1}{2}}}\bdrad(\la, \mu)^{n-1}}.
\end{equation}

The above equation is the main result of this paper, which we encapsulate into
the following theorem.

\begin{theorem}\label{thm:main}
  Fix a partition \(\lambda\) with \(n\) distinct parts, and a vector \(\mu\) in
  the relative interior of \(SH(\lambda)\).  Let
  $S_\la(x) = \frac{\det[\exp(x_i\la_j)]_{1\leq i,j \leq n}}{V(x)}, $ where
  $V(x)$ is the Vandermonde determinant. Then,
$$ \frac{\sqrt{(n-1)!} \cdot  \exp(-n)}{\lvol{\pSH(\la)}}
\leq \frac{\lvol{\GT(\la, \mu)}}{\inf_{x \in \R^n} S_{\lambda}(x) \exp(- x^T\mu)}
\leq \frac{\sqrt{(n-1)!}  \cdot \exp(n/2) \cdot \Gamma\inp{\frac{n+1}{2}}}{\inp{\sqrt{\pi} \cdot \bdrad(\la, \mu)}^{n-1} }.$$
\end{theorem}

\begin{remark}
  We derive here the guarantees claimed in \cref{eq-main-result} in the
  introduction.  Recall that in that setting, \(\lambda\) is an integral partition
  with \(n\) parts all whose entries are distinct, and \(\mu\) is an integer
  vector in the interior of \(\SH(\lambda)\).  In this setting, \cref{eq-bdrad-bound}
  shows that
  \( \bdrad(\lambda, \mu) \geq \frac{\gap}{16 \lambda_1 n^3} \geq \frac{1}{16 \lambda_1 n^3}.\) Further,
  \cref{prop-vol-psh} shows that
  \(\lvol{\pSH(\lambda)} \leq \lambda_1^{n-1}n^{2n}\).  Substituting these bounds in
  \cref{thm:main} gives the claim in \cref{eq-main-result}.
\end{remark}

\subsection{Computational aspects}
\label{sec:comp-aspects}
Given \cref{eq:opt-approx}, we need to minimize the convex function
$\tg_{\lambda, \mu}$ (see \cref{ob:project}) over $\R^{n-1}$.  To do so, we use the fact
that efficient approximate minimization of a convex function over a convex body
can be carried out with access to a weak membership oracle for the body: this is
Theorem 4.3.13 of \cite{GLSbook} (which, in turn, is based on results of Yudin
and Nemirovskii~\cite{YN76}), quoted as \cref{thm-gls-optimizer} below.  It
turns out that the two main ingredients required in order to apply this result
to our setting are (1) a restriction of the domain of optimization to a compact
convex set (carried out in \cref{prop:optimizer-setup}), and (2) designing an
efficient evaluation oracle for \(\tg_{\lambda, \mu}\) (carried out in
\cref{prop-compute-g}) .  We proceed to do this in this section.

In what follows, we will present the analysis of the algorithm  under the
following assumptions as stated in the abstract and the introduction:
\begin{equation}
  \begin{gathered}
    \text{\(\lambda\) is an \emph{integral} partition with \(n\) parts of distinct sizes},\\
    \text{the  entries of \({\lambda}\) are bounded above by a
    polynomial in \(n\), and,}\\
  \text{ \(\mu\) is an \emph{integer} point in the interior of \(\SH(\lambda)\).}
  \end{gathered}
\end{equation}
This is the same setting as was considered above in the context of bounding the
condition parameter \(\bdrad(\lambda, \mu)\) (in \cref{eq-bdrad-bound}).  We note,
however, that several of the technical results (including the results in
\cref{prop:optimizer-setup,prop-compute-g} mentioned above) in this section are
applicable more generally for rational $\lambda$ and $\mu$, and we specialize to the
above setting of integral values for \(\lambda\) and \(\mu\) only in
\cref{rem:R} and in the paragraphs following the proof of \cref{prop-compute-g}.  In
particular, using the results of this section, the analysis can be
extended to the general case of \(\lambda\) being any partition with distinct rational entries
and \(\mu\) an arbitrary rational vector in the relative interior of  \(\SH(\lambda)\), but the runtime bounds
would then have a more involved description, that would, in particular, have a dependence on the condition parameter \(\bdrad(\lambda, \mu)\), which would not necessarily be polynomial in $n$.

We begin with the following upper bound on $\inf\limits_{y \in \R^{n-1}}\tg(y)$.
\begin{observation}\label{ob:g-at-0}
  Let $\lambda, \mu \in \R^n$ be as in the statement of
  \cref{thm:main}.  Then
  \begin{equation}
  \label{eq-log-gt-vol-estimate}
  \tg_{\lambda,\mu}(\vec{0}) = \log \lvol{\tpGT(\lambda)}
  = \log \inp{
    \prod_{1 \leq i < j \leq n}
    \frac{\lambda_i - \lambda_j}{j - i}
  } \leq \binom{n}{2}\log \lambda_1.
  \end{equation}
\end{observation}
\begin{proof}
  By definition of $\tg_{\lambda, \mu}$,
  $\tg_{\lambda, \mu}(\vec{0}) = g_{\lambda, \mu}(\vec{0})$.  For the latter, we
  use its definition (\cref{eq:38}), followed by \cref{eq:tao-S-def} and then
  \cref{eq:gt-volume} to get the above claim.
\end{proof}
We now prove a lower bound on the value of $\hat{g}_{\lambda, \mu}(y)$ for $y$
far from the origin.
\begin{lemma}
  \label{lem:g-lower-bound}
  There exists an absolute constant $C$ such that the following is true for all
  $n \geq 4$.  Let $\lambda, \mu \in \R^n$ be as in the statement of
  \cref{thm:main}.  Let
  $\bdrad(\lambda, \mu)$ be as defined in
  \cref{eq-bdrad-def}.  Then, for all $y \in \R^{n-1}$ such
  that
  $\norm{y} \geq \frac{Cn^3}{\bdrad(\lambda, \mu)}\max\inb{1,
    \log\frac{1}{\bdrad(\lambda,\mu)}}$, it holds that
  \begin{equation}
    \tg_{\lambda, \mu}(y) \geq \frac{\bdrad(\lambda, \mu) \cdot \norm{y}}{\sqrt{n-1}}.
  \end{equation}
\end{lemma}
\begin{proof}
  Define $y' = (y_1, y_2, \dots, y_{n-1}, 0) \in \R^n$, and denote
  $\tilde{\mu} = p(\mu) \in \pSH(\lambda) \subseteq \R^{n-1}$.  Then
  $\tg_{\lambda,\mu}(y) = g_{\lambda,\mu}(y') = \log (S_{\lambda}(y')\exp(-y'\cdot\mu))$.  Let
  $B' = B(\tilde{\mu}, \bdrad(\lambda, \mu))$ be the $(n-1)$-dimensional ball of radius
  $\bdrad(\lambda, \mu)$ centered at $\tilde{\mu}$, as defined in the proof of
  \cref{prop:vol-upper-estimate}.  The computation leading to \cref{eq:46} in
  the proof of \cref{prop:vol-upper-estimate}, combined with the upper bound on
  $V_{\lambda,\mu}$ given in \cref{eq:53} then shows that for some absolute positive
  constant $C'$ (here \(r = r(\lambda, \mu)\) is as defined in \cref{eq-r-def})
  \begin{align}
    \tg_{\lambda,\mu}(y)
    &\geq
      \binom{n-1}{2}\log r - C'n^2 \log n +
      \log\int\limits_{\nu \in \SH(\lambda) \st p(\nu) \in B'}
      \exp(y' \cdot (\nu-\mu))
      \prod\limits_{1 \leq i \leq n - 1} d\nu_{i}\\
    &=
      \binom{n-1}{2}\log r - C'n^2 \log n +
      \log\int\limits_{\tilde{\nu} \in B'}
      \exp(y \cdot (\tilde{\nu}-\tilde{\mu}))
      \prod\limits_{1 \leq i \leq n - 1} d\tilde{\nu}_{i}
  \end{align}
  From the definition of $\bdrad(\lambda,\mu)$ in \cref{eq-bdrad-def}, we see that
  $\bdrad(\lambda,\mu) \leq r = r(\lambda, \mu)$.  The claim then follows by
  estimating the last integral using \cref{eq:68} in \cref{prop:ball-integral}
  and choosing $C$ to be a large enough constant.
\end{proof}

Combining \cref{ob:g-at-0,lem:g-lower-bound}, we obtain the following bounds on the
domain of optimization.
\begin{proposition}\label{prop:optimizer-setup}
  There exists an absolute constant $C$ such that the following is true for all
  $n \geq 4$.  Let $\lambda, \mu \in \R^n$ be as in the statement of
  \cref{thm:main}.  Let $\bdrad(\lambda, \mu)$ be as defined in
  \cref{eq-bdrad-def}, and define
  \begin{equation}
    R(\lambda, \mu) \defeq
    \frac{1}{\bdrad(\lambda, \mu)}
    \cdot
    \max\inb{
      \sqrt{n}
      \cdot \log \lvol{\tpGT(\lambda)},
      {Cn^3}
      \cdot
      \max\inb{
        1,
        \log\frac{1}{\bdrad(\lambda,\mu)}
      }
    }.
  \end{equation}
If
  $z \in \R^{n-1}$ is a minimizer of $\tg_{\lambda, \mu}$, then $z \in K$, where
  $K = K(\lambda, \mu)$ is the Euclidean ball of radius $R(\lambda, \mu) + 1$ in $\R^{n-1}$, centered
  at the origin.
\end{proposition}
\begin{proof}
  Let $C$ be as in the statement of \cref{lem:g-lower-bound}.  If
  $z \not\in K$, then $\norm{z} \geq R(\lambda, \mu) + 1$.  It then follows from the
  definition of $R(\lambda, \mu)$ combined with \cref{ob:g-at-0,lem:g-lower-bound} that
  $\tg_{\lambda,\mu}(z) > \tg_{\lambda,\mu}(\vec{0})$.  Thus $z$ cannot be a minimizer of
  $\tg_{\lambda,\mu}$.
\end{proof}

\begin{remark} \label{rem:R} Consider the case when $\la$ and $\mu$ are integer
  vectors, the entries of \(\lambda\) are polynomially bounded in $n$ and are
  all distinct, and $\mu$ lies in the relative interior of $\SH(\la)$.  Under
  these conditions, \cref{eq-bdrad-bound} shows
  that the condition parameter $\frac{1}{\bdrad(\la, \mu)}$ is also bounded above by a
  polynomial in $n$, and hence (using \cref{eq-log-gt-vol-estimate}) it follows
  that there is an absolute constant $C$ such that $R(\la, \mu) < n^C$.
\end{remark}

Optimization of $\tg_{\lambda,\mu}$ also requires an efficient additive approximation
for $\tg_{\lambda,\mu}$.  As we show below, the following proposition will be sufficient
for our needs.  We first set up some standard notation.

\paragraph{\textbf{Notation.}} Given a rational number \(x\), we denote by \(\rep{x}\)
the number of bits required to specify the numerator and denominator of the
reduced form of \(x\) in binary.  We also extend this notation to sets and
tuples of rational numbers, so that, for a finite tuple $S$ of rational numbers,
$\rep{S}$ denotes the total representation length of all the rational numbers in
$S$.  By \(\poly{\dots}\) we mean a polynomial of its arguments, with the degree
and the coefficients of the polynomial being absolute constants; however,
different occurrences of \(\poly{\cdot}\) denote possibly different polynomials.

\begin{proposition}[Approximation of $\tg_{\lambda,\mu}$]\label{prop-compute-g}
There exists a deterministic algorithm with the following guarantee.

\medskip
\noindent\textbf{INPUT:}
\begin{itemize}
  \item A rational accuracy parameter $0 < \delta < 1$.
  \item A positive rational \(n\)-tuple \(\lambda\) satisfying
    $ \lambda_1-(n-1) \ge \lambda_2-(n-2) \ge \cdots \ge \lambda_{n-1}-1 \ge
    \lambda_n > 0.  $ (In other words, \(\gap \geq 1\).)
  \item A rational vector $\mu \in \SH(\lambda)$.
  \item A rational vector $y \in \mathbb{R}^{n-1}$.
\end{itemize}

\medskip
\noindent\textbf{OUTPUT:}
A rational number $t$ such that
$
  \bigl| \tg_{\lambda,\mu}(y) - t \bigr| \le \delta .
$

\medskip
\noindent\textbf{RUNTIME:}
The algorithm runs in time polynomial in
\(n, \|\lambda \|, \|y\|, \text{and } \rep{\lambda,\mu,y,\delta}\), where
$\rep{\cdot}$ denotes the binary encoding length of the input.
\end{proposition}

\begin{remark}
  For simplicity, the above result is stated for the case \(\gap \geq 1\).
  However, using the identity (see \cref{eq-scale-shift-s})
  \( S_{\lambda}(x) = \gap^{\binom{n}{2}}\cdot S_{\lambda/\gap}(\gap x)\), the
  same proof shows that this assumption can be dropped at the cost of a further
  polynomial dependence on \(\max\inb{1, 1/\gap}\) in the runtime.
\end{remark}
\begin{proof}[Proof of \cref{prop-compute-g}]

By \cref{ob:project}, $\hat{g}_{\la, \mu} (y) = g_{\la, \mu}(x)$ for any $x$ such
  that $y = q(x)$. For specificity, we choose $x$ to be the vector obtained by
  appending $0$ to $y$, so that $\norm{x} = \norm{y}$.  In order to evaluate
  $g_{\la, \mu}(x)$ to within accuracy $\pm \de$, it suffices to evaluate
  $\log S_\la(x)$ to within $\pm\de/2$, and $x^\top\mu$ to within $\pm\de/2$. Since
  the latter is easy, we focus on evaluating $\log S_\la(x)$ to within
  $\pm c\de$, where $c$ is some positive absolute constant.
It turns out that the main algorithmic issues to handle are (1) $x$
  potentially having non-distinct entries, which makes the determinant formula
  in \cref{eq:s-det-unequal} inapplicable, and (2) careful accounting for the
  representation lengths of exponential entries in the determinant appearing in
  \cref{eq:s-det-unequal}.

  For the first requirement above, we perturb \(x\) slightly.  The calculation
  in \cref{eq:15} shows that $\nabla \log S_{\lambda}(x)$, being a weighted mean of vectors
  in $\SH(\lambda)$, lies in $\SH(\lambda)$, so that $\log S_{\lambda}$ is an
  $L$-Lipschitz function for any
  $L \geq \max_{v \in \SH(\la)}\norm{v} = \norm{\lambda}$.  To keep $L$ integral, we make
  the concrete choice $L = \ceil{\norm{\lambda}}$. Note that $S_\la(x)$ is a symmetric
  function of $x$. So without loss of generality suppose that $x_1 \geq \dots \geq x_n$. Let
  \begin{equation}
    \label{eq-x-hat-def}
    \hat{x} = (x_1 + \frac{(n-1)\de}{2n^2L}, x_2 + \frac{(n-2)\de}{2n^2L}, \dots,
    x_n).
  \end{equation}
  Then,
  \begin{enumerate}
  \item The entries of \(\hat{x}\) are strictly decreasing, and the gap between
    successive entries is at least \(\frac{\delta}{2n^2L}\).
  \item \(\norm{x - \hat{x}} < \frac{\delta}{4L} < 1/4\), so that
    $|\log S_\la(\hat{x}) - \log S_\la(x)| < \de/4$.
  \end{enumerate}
  Since \(\hat{x}\) has distinct entries, we can now use the determinant formula
  \begin{equation}
    \label{eq-S-compute-approx}
    S_{\lambda}(\hat{x}) = \frac{\det[\exp(\hat{x}_i\lambda_j)]_{1 \leq i, j \leq n}}{V(\hat{x})},
  \end{equation}
  where \(V\) represents the Vandermonde determinant.  We first note that
  \(\log V(\hat{x})\) can be approximated to an additive accuracy of
  \(\pm\delta/8\) using at most
  \(\poly{n, \rep{\lambda}, \rep{y}, \rep{\delta}}\) bit operations (see
  \cref{lem-log-approx}).

  It remains to show that the same can be done for the determinant
  \(\det[\exp(\hat{x}_i\lambda_j)]_{1 \leq i, j \leq n}\).  
We
  begin by establishing a lower bound on the absolute value of this determinant.

  For this, we let \(T\) denote \(2n(n-1)\) times the least common multiple of
  the denominators of the \(\lambda_i\).  Note that \(T\) can be computed in
  \(\poly{n, \rep{\lambda}}\) bit operations and \(\log T\) is bounded above by
  \(\poly{n, \rep{\lambda}}\).  Consider the vector
  \begin{equation}
    \label{eq:21}
    \hat{\la} \defeq (T{\la}_1 - (n-1), T{\la}_2 - (n-2),
    \dots, T{\la}_{n-1} -1, T{\la}_n).
  \end{equation}
  Note that \(\hat{\lambda}\) is a non-negative integer tuple with gaps of at least
  $1$ between successive entries.  From \cref{eq-jacobi-trudi-ratio-of-dets}, we
  then see that we can relate our determinant to the Schur polynomial as
  follows:
  \begin{equation}
    \label{eq:23}
    s_{\hat{\la}}(\exp\inp{\hat{x}/T}) =
    \frac{\det[\exp(\hat{x}_i{\la}_j)]_{1\leq i,j \leq
        n}}{V(\exp(\hat{x}/T))}.
  \end{equation}
  We now recall that the coefficients of the expansion of the Schur polynomial
  $s_{\hat{\la}}(z)$ using the monomial basis (\cref{eq-schur-poly}) are all
  non-negative, and the coefficient of $z^{\hat{\la}}$ in this expansion equals
  $1$.  This observation implies that
  $$s_{\hat{\la}}(e^{\hat{x}/T}) \geq \exp(\hat{x}\cdot \hat{\la}/T),$$ and
  consequently,
  \begin{equation}
    \label{eq:crux}
    \det[\exp(\hat{x}_i{\la}_j)]_{1\leq i,j \leq
      n} \geq \exp(\hat{x}\cdot \hat{\la}/T) V(\exp(\hat{x}/T)) \defeq v.
  \end{equation}
  We first show that \(v\) is not too small.  In particular, we have
  \begin{align}
    \abs{\log v}
    &\leq \frac{\norm{\hat{x}}\norm{\hat{\lambda}}}{T}
      + \sum_{i=1}^{n-1}\frac{(n-i)\abs{\hat{x}_i}}{T}
      + \sum_{i=1}^{n-1}\sum_{j = i + 1}^n\abs{\log (1 - \exp(-(\hat{x}_i - \hat{x}_j)/T))}\\
    &\leq \frac{\norm{\hat{x}}\norm{\hat{\lambda}}}{T} +
      \frac{n\sqrt{n}\norm{\hat{x}}}{T} +
      \binom{n}{2}\abs{\log (1 - \exp(-\delta/(2n^2LT)))},\label{eq-v-prelim-bound}\\
    &\qquad\qquad
      \text{ since }
      \hat{x}_i - \hat{x}_j \geq \frac{\delta}{2n^2L} \text{ for } i < j.\nonumber
  \end{align}
  Now note that
  \(\norm{\hat{\lambda}}/T \leq \norm{\lambda} + \frac{n(n-1)}{2T} \leq \norm{\lambda} + 1/4\), and
  also that \(\norm{\hat{x}} < \norm{x} + 1 = \norm{y} + 1\).  From
  \cref{lem-log-one-minux-exp-approx}, we thus get that
  \begin{equation}
    \abs{\log v} \leq (1+\norm{y})(1 + \norm{\lambda}) + n^2\log\inp{\frac{4n^2\abs{\lambda}T}{\delta}}.
  \end{equation}
  In particular, this implies that we can construct a positive integer \(v'\)
  satisfying \(v' \geq 1/v \) in
  \(\poly{n, \norm{\lambda}, \norm{y}, \rep{y, \lambda, \delta}}\) bit operations.

  Now, define \(\tau \defeq \max_{1 \leq i, j \leq n} |\hat{x}_i {\la}_j|\).  Note that
  \(\tau\) is bounded above by a polynomial in \(\abs{\lambda}\) and
  \(\abs{y}\), and also that \(\rep{\tau}\) is at most
  \(\poly{\rep{y, \lambda, \delta}}\).  Thus, using
  \(\poly{n, \norm{y}, \norm{\lambda}, \rep{y, \lambda, \delta}}\) bit operations we can
  construct a positive integer \(D'\) such that
  \begin{equation}
    D' \geq 2n \cdot \max\inp{1, \frac{32\cdot n! \cdot \tau^n \cdot v'}{\delta}}.
  \end{equation}
  Using standard approximations for the exponential function (see
  \cref{lem-exp-approx}) we then construct, using at most
  \(\poly{\tau, \rep{\tau, D'}} \leq \poly{n, \norm{y}, \norm{\lambda}, \rep{y,
      \lambda, \delta}}\) bit operations, an \(n \times n\) matrix \(M'\) such
  that for each \(i, j\), \(M'_{i,j}\) is an \((1/D')\)-relative
  approximation\footnote{Given \(\epsilon > 0\), we say that a quantity
    \(\tilde{v}\) is an \emph{\(\epsilon\)-relative approximation} of a quantity
    \(v\) if \(\abs{\tilde{v} - v} \leq \epsilon\abs{v}\).} of
  \(\exp(\hat{x}_i\lambda_j)\).  Note that by construction,
  \(\rep{M'} \leq \poly{n, \norm{y}, \norm{\lambda}, \rep{y, \lambda,
      \delta}}\).  As described in \cref{lem-det-approx}, \(\det(M')\) can thus
  be computed in
  \(\poly{n, \norm{y}, \norm{\lambda}, \rep{y, \lambda, \delta}}\) bit
  operations and provides a \((\delta/16)\)-relative approximation to
  \(\det[\exp(\hat{x}_i{\la}_j)]_{1\leq i,j \leq n}\).  Taking its logarithm to
  \(\pm \delta/32\) additive accuracy (see \cref{lem-log-approx}) gives a
  \(\pm\delta/8\) additive approximation to
  \(\log \det[\exp(\hat{x}_i{\la}_j)]_{1\leq i,j \leq n}\), again using
  \(\poly{n, \norm{y}, \norm{\lambda}, \rep{y, \lambda, \delta}}\) bit
  operations.

  In combination with the discussion following
  \cref{eq-x-hat-def,eq-S-compute-approx}, the procedure outlined above gives a
  \(\pm\delta/2\) additive approximation to \(\log S_{\lambda}(x)\) using at most
  \(\poly{n, \norm{y}, \norm{\lambda}, \rep{\lambda, y, \delta}}\) bit operations.  Since
  \(\tg_{\lambda, \mu}(y) = S_{\lambda}(x) - x^\top\mu\), this completes the proof.
\end{proof}

We now specialize to the setting of \cref{rem:R}.  Using
\cref{prop:optimizer-setup,prop-compute-g} we can apply Theorem 4.3.13 of
\cite{GLSbook} (based on results of \cite{YN76}) to minimize
$\tg_{\lambda,\mu}$ over the convex body
$K({\lambda,\mu})$ (defined in \cref{prop:optimizer-setup}) and get an
$\epsilon$-additive approximation of the minimum value of $\tg_{\la,
  \mu}$ in time polynomial in
$n$ and the binary representation lengths of $\epsilon$, $\lambda$ and
$\mu$. We describe Theorem 4.3.13 of \cite{GLSbook} below, followed by details of the
inputs passed to the algorithm promised by it.  We use the following notation of
\cite{GLSbook}.  When \(r > 0\), we denote by $S(\alpha,
r)$ the set of all points within Euclidean distance $r$ of
$\alpha$, where
$\alpha$ can be a point or a set of points.  When \(r < 0\), and \(\alpha\) is a
set, \(S(\alpha, r)\) denotes the set of those points \(\beta\) for which
\(S(\beta, -r)\) is contained in \(\alpha\).

\begin{theorem}[\textbf{\cite[Theorem
    4.3.13]{GLSbook}}] \label{thm-gls-optimizer} There exists an
  oracle-polynomial time algorithm that solves the following problem: \\
  \textbf{INPUT:}  A
  rational number $\epsilon > 0$, a convex body $K \subseteq \R^m$, along with a point
  \(a_0 \in \R^m\) and positive reals \(r < R\) such that
  $S(a_0, r) \subseteq K \subseteq S(a_0, R)$, given by a weak membership oracle, and a convex
  function $f: \R^m \ra \R$ given by an oracle that, for every $x \in \Q^m$ and
  $\de > 0$, returns a rational number $t$ such that $|f(x) -t| \leq \de.$
\\
  \textbf{OUTPUT:} A vector $y \in S(K, \eps)$ such that $f(y) \leq f(x) + \eps$ for all $x \in  S(K, -\eps)$.\\
   \textbf{RUNTIME:} The algorithm uses at most \(\poly{\rep{a_0, r, R, \eps}}\)
   bit operations provided that the weak  membership oracle for $K$ and the
   evaluation oracle for $f$ can be implemented to run in time polynomial in the
   representation lengths of their inputs..
\end{theorem}

In order to apply the above theorem in the setting of \cref{rem:R}, we need to
specify the function $f$, which we now do. As seen in the preceding proof,
\(\tg_{\lambda, \mu}\) is \(L\)-Lipschitz with
\(L = \ceil{\norm{\lambda}} \geq 1\).  From \cref{ob:g-at-0}, we can efficiently
compute a lower bound \(v_0\) for $\tg_{\la, \mu}(0)$ satisfying
\(v_0 \leq \tg_{\lambda, \mu}(0) \leq v_{0} + 1\).  It also follows from
\cref{ob:g-at-0} that this \(v_{0}\) satisfies
\(\abs{v_{0}} \leq {n^2}\cdot(\log n + \abs{\log \lambda_1})\).  For
\(C_1 = 3(R(\la, \mu) + 2)\) (where \(R(\lambda, \mu) \) is as defined in
\cref{prop:optimizer-setup}) and $C_2 = 2L$, we take $f$ to be defined on
$\R^m = \R^{n-1}$ by
\begin{equation}
  \label{eq:14}
  f(x) := \max(\tg_{\la, \mu}(x), v_0 - C_1L + C_2 \|x\|).
\end{equation}
Note that $f$ is convex.  We take $K$ to be the ball of radius
$R(\la, \mu) + 3/2$ centered at the origin, $a_0 = 0$ and
$r = R = R(\la, \mu) + 3/2$.  Then, we see that on $S(K, 1/2)$ (which is just the
ball of radius \(R(\lambda, \mu) + 2\) centered at the origin), $f$ equals
$\tg_{\la, \mu}$, while outside the ball of radius \(3(R(\la, \mu) + 2) + 1\) centered
at the origin, $f$ equals $v_0 - C_1 + C_2 \|x\|_2$.  Recall that in the setting
of \cref{rem:R}, \(R(\lambda, \mu)\) is bounded above by \(n^{C}\) for some fixed
constant \(C\).  Thus, the evaluation of \(f\) requires the evaluation of
\(\tg_{\lambda, \mu}(y)\) only for \(y\) within a origin-centered ball of radius
\(3(n^{C} + 1)\) and hence, by \cref{prop-compute-g}, can be implemented using
\(\poly{n, \rep{\lambda, \mu, x, \delta}}\) bit operations (since, in the setting of
\cref{rem:R}, the entries of \(\lambda\) are bounded above by a polynomial in \(n\)).

Suppose now that the algorithm promised in \cref{thm-gls-optimizer} is
instantiated with the above \(f\) and \(K\) and accuracy parameter
\(\epsilon/2\) with \(\epsilon \in (0, 1/2)\).  Then, since the global minimizer
of \(\tg_{\lambda, \mu}\) is contained in \(S(K, -1/2)\) (by
\cref{prop:optimizer-setup}), and since \(f\) agrees with \(\tg_{\lambda,\mu}\)
on \(S(K, 1/2)\), the guarantee provided in \cref{thm-gls-optimizer} implies
that the algorithm outputs a vector \(y \in S(K, 1/2)\) such that
\(f(y) = \tg_{\lambda, \mu}(y) \leq \tg_{{\lambda, \mu}}^{*} + \epsilon/2\),
where \(\tg_{\lambda,\mu}^{*}\) is the minimum value of \(\tg_{\lambda, \mu}\).
Further, the algorithm runs in time \(\poly{n, \rep{\epsilon}}\), since when
\(\lambda\) and \(\mu \in \SH(\lambda)\) are integral with entries of
\(\lambda\) bounded above by a polynomial in \(n\), the representation lengths
\(\rep{\lambda}\) and \(\rep{\mu}\) are also at most \(\poly{n}\).  Note also
that \(\norm{y} \leq R(\lambda, \mu)\) and \(\rep{y}\) are also bounded above
by a polynomial in \(n\).  Using the algorithm in \cref{prop-compute-g} to
evaluate \(\tg_{\lambda, \mu}(y)\) to accuracy \(\pm\epsilon/2\), we thus have
the following theorem regarding the computation of \(\tg_{\lambda,\mu}^{*}\) and
hence of the estimate in \cref{eq:opt-approx}.
\begin{theorem} \label{thm-main-algo} Let \(\kappa\) be any positive integer.  Then,
  there exists an
  algorithm that solves the following problem with the runtime mentioned below. \\
  \textbf{INPUT:}
\begin{itemize}
  \item A rational accuracy parameter $0 < \epsilon < 1/2$.
  \item \(\lambda\) is an \emph{integral} partition with \(n\) parts of distinct
    sizes, where the magnitude of the entries of \(\lambda\) are bounded above
    \(\kappa\cdot n^{\kappa}\), and,
  \item  \(\mu\) is an \emph{integral} point in the interior of \(\SH(\lambda)\).
\end{itemize}
\textbf{OUTPUT:} A number \(v\) satisfying
\begin{equation}
  \label{eq:12}
  1 - \epsilon \leq \frac{v}{\inf_{x \in \R^n} S_{\lambda}(x) \exp(- x^T\mu)} \leq 1 + \epsilon.
\end{equation}
\textbf{RUNTIME:} \(\poly{n, \rep{\eps} }\) bit operations.
\end{theorem}
\begin{proof}
  Recall (from \cref{prop:log-convex-cont-schur,ob:project,eq:opt-approx}) that
  \(\inf_{x \in \R^n} S_{\lambda}(x) \exp(- x^T\mu) = \exp(\tg_{\lambda,
    \mu}^{*}\)).  Now, the discussion above shows that we can compute in
  \(\poly{n, \rep{\epsilon}}\) bit operations a rational number \(u\) such that
  \(\tg_{\lambda, \mu}^{*} \leq u \leq \tg_{\lambda, \mu}^{*} + \eps/4\), and
  which satisfies (due to the Lipschitz property of \(\tg_{\lambda, \mu}\)
  discussed in the preceding paragraphs and the bound on
  \(\tg_{\lambda, \mu}(0)\) given in \cref{ob:g-at-0})
  \(\abs{u} \leq \poly{n}\).  We now use \cref{lem-exp-approx} to compute \(v\)
  to be an \(\frac{\epsilon}{4}\)-relative approximation to \(\exp(u)\), in time
  \(\poly{n, \rep{\epsilon}}\).
\end{proof}

\appendix

\section{Jacobian computations}
In this section, we collect for easy reference some standard Jacobian
computations that are required in the application of the area and the co-area
formulas in the main text of the paper.

\paragraph{\textbf{Notation.}} Let \(m\) and \(n\) be positive integers.  We
denote by \(\vec{0}_{m \times n}\) (respectively, \(\vec{1}_{m\times n}\)) the
\(m \times n\) matrix all whose entries are \(0\) (respectively, \(1\)).  We
denote by \(\vec{I}_{m \times m}\) the \(m \times m\) identity matrix.

\subsection{\texorpdfstring{Computations for
    \cref{eq:schul-vol-rel,eq:vol-relation}}{Computations for applications of the
    area formula}}
\label{sec:comp-cref-vol}

For any real number $c$ and a positive integer $k$, define the map
$\beta_{c,k}: \R^{k} \rightarrow \R^{k+1}$ by
\begin{equation}
  \label{eq:27}
  \beta_{c,k}(\vec{x}) =
  \begin{pmatrix}
    \vec{x}\\
    c - \sum\limits_{i=1}^kx_i
  \end{pmatrix}.
\end{equation}
We then have
\begin{equation}
  \label{eq:28}
  D\beta_{c,k}(\vec{x})
  = M_k \defeq
  \begin{pmatrix}
    \vec{I}_{k \times k}\\
    -\vec{1}_{1\times k}
  \end{pmatrix}.
\end{equation}
Thus,
\begin{equation}
  D\beta_{c,k}(\vec{x})^\top D\beta_{c,k}(\vec{x}) = M_k^{\top}M_k =
  \vec{I}_{k \times k} +
  \vec{1}_{k \times k},
\end{equation}
which has eigenvalues $(k+1)$ (with multiplicity $1$) and $1$ (with multiplicity
$k-1$).  Thus we get
\begin{equation}
  \label{eq:30}
  J_{\beta_{c,k}}(\vec{x}) \defeq \sqrt{\det\inp{  D\beta_{c,k}(\vec{x})^\top
      D\beta_{c,k}(\vec{x})}} = \sqrt{\det\inp{{M_k}^{\top}M_k}} = \sqrt{k+1}.
\end{equation}

We now apply the above computation to justify
\cref{eq:vol-relation,eq:schul-vol-rel}. We begin with
\cref{eq:schul-vol-rel}. Let $\lambda$ be as in the setting of
\cref{eq:schul-vol-rel}, and consider the map
$f \defeq \beta_{\abs{\lambda}, n - 1}$.  Then $f$ is a bijection from
$\pSH(\lambda)$ to $\SH(\lambda)$.  The area formula~\cite[Theorem 3.2.3]{F96} then gives
\begin{equation}
  \H^{n-1}(\SH(\lambda))
  = \int\limits_{\pSH(\lambda)}
  J_f(x) \prod_{i=1}^{n-1}dx_{i}
  \overset{\textup{\cref{eq:30}}}{=} \sqrt{n}\cdot \lvol{\pSH(\lambda)},
\end{equation}
which proves \cref{eq:schul-vol-rel}.

We now turn to \cref{eq:vol-relation}.  Let $\lambda, \mu$ be as in that
equation.  Consider the map
$\gamma: \R^{(n-1)(n-2)/2} \rightarrow \R^{n(n-1)/2}$ given by
\begin{equation}
  \gamma\inp{
    \inp{x_{i\,j}}_{
      \substack{
        2 \leq i \leq n -
        1\\1\leq j \leq i - 1}
    }} \defeq
  \inp{x_{i\,j}}_{
    \substack{
      1 \leq i \leq n - 1
      \\1\leq j \leq i}
  },
\end{equation}
where for $1 \leq i \leq n -1$, $x_{i\, i}$ is set to
$\sum_{j=1}^i\mu_j - \sum_{j=1}^{i-1}x_{i\,j}$.  Note that $\gamma$ is a
bijection from $\pGT(\lambda, \mu)$ to $\GT(\lambda, \mu)$.  Thus, the area
formula~\cite[Theorem 3.2.3]{F96} yields that
\begin{equation}
  \label{eq:33}
  \H^{(n-1)(n-2)/2}(\GT(\lambda,\mu))
  = \int\limits_{
    \pGT(\lambda,\mu)
  }
  J_{\gamma}\inp{
    \inp{x_{i\,j}}_{
      \substack{
        2 \leq i \leq n -
        1\\1\leq j \leq i - 1}
    }
  }
  \prod\limits_{
    \substack{
      2 \leq i \leq n -
      1\\1\leq j \leq i - 1
    }} dx_{i\,j}.
\end{equation}
It remains to compute $J_{\gamma}$.  We begin by observing that the map $\gamma$
can be written in terms of the functions $\beta_{c,k}$ defined in \cref{eq:27},
as follows
\begin{equation}
  \gamma\inp{
    \inp{x_{i\,j}}_{
      \substack{
        2 \leq i \leq n -
        1\\1\leq j \leq i - 1}
    }} =
  \hat{\mu}_1 \oplus \bigoplus\limits_{i=2}^{n-1}\beta_{\hat{\mu}_i, i -
    1}\inp{\inp{x_{i\,j}}_{1 \leq j \leq i - 1}},
\end{equation}
where, for each $i$, $\hat{\mu_i}$ denotes the running sum $\sum_{j=1}^i\mu_j$.
We can thus write the first derivative of $\gamma$ as a
$\frac{n(n-1)}{2} \times \frac{(n-1)(n-1)}{2}$ matrix as follows.
\begin{align}
  D\gamma\inp{
    \inp{x_{i\,j}}_{
      \substack{
        2 \leq i \leq n -
        1\\1\leq j \leq i - 1}
  }}
  &=
    \begin{pmatrix}
      \vec{0}_{1 \times (n-1)(n-2)/2}\\
      \bigoplus\limits_{i=2}^{n-1}
      D\beta_{\hat{\mu}_i, i - 1}\inp{\inp{x_{i\,j}}_{1 \leq j \leq i - 1}}
    \end{pmatrix}\\
  &\overset{\textup{\cref{eq:28}}}{=}
    \begin{pmatrix}
      \vec{0}_{1 \times (n-1)(n-2)/2}\\
      \bigoplus\limits_{i=2}^{n-1}M_{i-1}
    \end{pmatrix}.
\end{align}
We thus get that
\begin{equation}
  \label{eq:25}
  D\gamma\inp{
    \inp{x_{i\,j}}_{
      \substack{
        2 \leq i \leq n -
        1\\1\leq j \leq i - 1}
  }}^{\top}
  D\gamma\inp{
    \inp{x_{i\,j}}_{
      \substack{
        2 \leq i \leq n -
        1\\1\leq j \leq i - 1}
  }}
  = \bigoplus\limits_{i=2}^{n-1}(M_{i-1}^{\top}M_{i-1}).
\end{equation}
Thus,
\begin{equation}
  \begin{aligned}[c]
    J_{\gamma}\inp{\inp{
    \inp{x_{i\,j}}_{
    \substack{
    2 \leq i \leq n -
    1\\1\leq j \leq i - 1}
    }}}
    &=
      \sqrt{\det\inp{
      D\gamma\inp{
      \inp{x_{i\,j}}_{
      \substack{
      2 \leq i \leq n -
      1\\1\leq j \leq i - 1}
    }}^{\top}
    D\gamma\inp{
    \inp{x_{i\,j}}_{
    \substack{
    2 \leq i \leq n -
    1\\1\leq j \leq i - 1}
    }}}}\\
    &\overset{\textup{\cref{eq:25}}}{=}
      \sqrt{\prod\limits_{i=2}^{n-1}\det\inp{M_{i-1}^{\top}M_{i-1}}}\\
    &\overset{\textup{\cref{eq:30}}}{=}
      \sqrt{(n-1)!}.
  \end{aligned}
\end{equation}
Substituting this in \cref{eq:33} yields \cref{eq:vol-relation}.

\section{Some useful estimates}
In this section we collect some standard computations. The only equation used from this section in the main body of the paper is  \cref{eq:68} in \cref{prop:ball-integral} which is used in the proof of Lemma~\ref{lem:g-lower-bound}.
\begin{proposition}
  Assume $n \geq 2$, and define the function
  \begin{equation}
    \label{eq:54}
    F_n(\alpha) \defeq \int\limits_{-1}^1 \exp(\alpha t) \cdot (1-t^2)^{n/2}dt.
  \end{equation}
  Then for every $\gamma > 0$
  \begin{equation}
    \label{eq:55}
    F_n(n\gamma) \geq \frac{e^{\gamma\sqrt{2n}}(\gamma\sqrt{2n} - 1) - (n\gamma^2-1)}{n^2\gamma^3}.
  \end{equation}
  Thus, there exists a positive constant $C$ such that if $\tau \geq
  C\sqrt{n}$ ,then
  \begin{equation}
    \label{eq:57}
    F_{n}(\tau) \geq \frac{\tau}{n}\exp(1.4 \,\tau/\sqrt{n}).
  \end{equation}
\end{proposition}

\begin{proof}
  For $\alpha > 0$ and $n \geq 2$, we have
  \begin{align}
    F_n(\alpha)
    & \geq \int\limits_0^{\sqrt{\frac{2}{n}}} \exp(\alpha t)
      (1-t^2)^{n/2} dt
      \geq
      \int\limits_0^{\sqrt{\frac{2}{n}}} \exp(\alpha t)
      \inp{1 - \frac{nt^2}{2}} dt\\
    &= \frac{e^{\alpha\sqrt{2/n}}(\alpha\sqrt{2n} - n)  - (\alpha^2-
      n)}{\alpha^3}. \label{eq-app-f-estimate}
  \end{align}
  \Cref{eq:55} now follows by setting $\alpha = \gamma n$.  To obtain
  \cref{eq:57}, we write \(\tau = A\sqrt{n}\) and use \cref{eq-app-f-estimate}
  to get
  \begin{align}
    \label{eq:11}
    F_n(\tau)
    &\geq \frac{
      \exp(A\sqrt{2})\cdot(A\sqrt{2} - 1) - (A^2 - 1)
      }{
      A^3\sqrt{n}
      }\\
    &= \frac{\tau}{n}\exp(1.4\tau/\sqrt{n})\cdot\insq{
      \frac{
      (A\sqrt{2} - 1)\cdot\exp((\sqrt{2} - 1.4)A) - (A^{2} - 1)\exp(-1.4 A)
      }{A^{4}}}.
  \end{align}
  Since \(\sqrt{2} - 1.4 > 0\), there exists a positive constant \(C\) such if
  \(\tau/\sqrt{n} = A \geq C\), then the quantity in the square brackets above
  is at least \(1\).  This establishes \cref{eq:57}.
\end{proof}

\begin{proposition}
  \label{prop:ball-integral} Assume $n \geq 3$. Let $B(v, r)$ denote the
  Euclidean ball of radius $r$ centered at $v \in \R^n$.  Let $y \in \R^n$ be an
  arbitrary vector.  Let $b_{n-1}$ denote the volume of the Euclidean unit ball
  in $\R^{n-1}$.  Then,
  \begin{equation}
    \label{eq:65}
    \log\int\limits_{z \in B(v, r)}\exp(y\cdot (z-v)) dz = n\log r
    + \log b_{n-1} + \log F_{n-1}(r\norm[2]{y}).
  \end{equation}
  Further, there exists an absolute constant $C > 0$, such that if
  $\norm{y} \geq \frac{Cn}{r}$ then
  \begin{equation}
    \label{eq:69}
    \log\int\limits_{z \in B(v, r)}\exp(y\cdot (z-v)) dz
    \geq
    n\log r
    + \log b_{n-1}+
    \frac{1.4\,r\norm{y}}{\sqrt{n-1}},
  \end{equation}
  and if $\norm{y} \geq \frac{Cn^3}{r}\max\inb{1, \log\frac{1}{r}}$ then
  \begin{equation}
    \label{eq:68}
    \log\int\limits_{z \in B(v, r)}\exp(y\cdot (z-v)) dz
    \geq
    \frac{1.2\,r\norm{y}}{\sqrt{n}}.
  \end{equation}
\end{proposition}
\begin{proof}
  A change of variables from $z$ to $(z-v)/r$ gives
  \begin{equation}
    \label{eq:66}
    \log\int\limits_{z \in B(v, r)}\exp(y\cdot (z-v)) dz
    = n\log r
    + \log\int\limits_{z \in B(0, 1)}\exp(r \, y\cdot z) dz.
  \end{equation}
  Let $\hat{y}$ be the unit vector in the direction of $y$.  Then, the map
  $z \mapsto z\cdot \hat{y}$ from $\R^n$ to $\R$ has Jacobian equal to $1$, so
  that the co-area formula applied with this map gives
  \begin{equation}
    \label{eq:67}
    \int\limits_{z \in B(0, 1)}\exp(r\,y\cdot z) dz
    = \int\limits_{t \in [-1,1]} \exp(rt\norm[2]{y})
    \cdot \H_{n-1}(\inb{z \in B(0, 1) \st \hat{y}\cdot z = t}) dt.
  \end{equation}
  The set $\inb{z \in B(0, 1) \st \hat{y}\cdot z = t}$ is an $(n-1)$-dimensional Euclidean
  ball of radius $\sqrt{1 - t^2}$.  Thus, its $(n-1)$-dimensional Hausdorff
  measure is $(1-t^2)^{(n-1)/2}b_{n-1}$.  Substituting this in \cref{eq:67}, and
  then the latter in \cref{eq:66}, and then comparing with the definition of $F$
  in \cref{eq:54} gives the claim in \cref{eq:65}.  \Cref{eq:68,eq:69} then
  follow from \cref{eq:57} after choosing $C$ large enough, and an application
  of the Stirling approximation to upper bound $- \log b_{n-1}$ as
  $O(n \log n)$.
\end{proof}

\section{Computational approximation theory}
In this section, we collect computational aspects of some standard facts from
approximation theory.  These are used in the proofs in \cref{sec:comp-aspects}.
We use the notation introduced just before \cref{prop-compute-g}.

\begin{lemma}[\textbf{Approximating the exponential}]
  \label{lem-exp-approx}
  There is an algorithm with the following characteristics.\\
  \textbf{INPUT:} A rational number \(x \geq 0\) and a rational accuracy parameter
  \(\delta \in (0, 1)\).\\
  \textbf{OUTPUT:} A rational number \(t\) satisfying
  \begin{equation} \abs{\exp(x) - 1/t} \leq \delta \exp(x) \quad \text{ and } \quad \abs{\exp(-x) - t} \leq \delta \exp(-x).
  \end{equation}
  \textbf{RUNTIME:} The algorithm uses at most \(\poly{x, \rep{x, \delta}}\) bit
  operations.
\end{lemma}
\begin{proof}
  It is a well-known result in approximation theory~\cite[Lemma
  1]{codyChebyshevRationalApproximations1969} that the polynomial
  \(S_{\ell}(x) \defeq \sum_{i = 0}^{\ell}\frac{x^i}{i!}\) satisfies
  \begin{equation}
    \label{eq:13}
    0 \leq \frac{1}{S_{\ell}(x)} - \exp(-x) \leq 2^{-\ell} \text{ for all integers } \ell \geq 0.
  \end{equation}
  Choose \(\ell\) to be an integer that is at least
  \(2\ceil{\log(1/\delta) + x}\).  Then \(t = 1/S_{\ell}(x)\) satisfies the claimed
  conditions.  Further, the definition of \(S_{\ell}(x)\) given above shows that it
  can be evaluated using \poly{\ell} additions and multiplications, each of which
  have output bit length bounded above by \(\poly{\ell, \rep{x}}\).  This completes
  the proof.
\end{proof}

\begin{lemma}[\textbf{Approximating the determinant with approximate entries}]
  \label{lem-det-approx}
  There exists an algorithm with the following characteristics.\\
  \textbf{INPUT:} A rational \(n \times n\) matrix \(M'\), along with the
  guarantee that there exist a real \(n \times n\) matrix \(M\) and rational numbers
  \(\delta' \in (0, 1/(2n))\), \(\tau\) and \(v\) such that
  \(\abs{\det M} \geq v\) and for all \(1 \leq i, j \leq n\),
  \(\abs{M'_{i,j} - M_{i,k}} \leq (\delta'/2)\abs{M_{i,j}}\) and \(\abs{M_{i,j}} \leq \tau\).\\
  \textbf{OUTPUT:} A rational number \(t\) satisfying
  \begin{equation}
    \label{eq:18}
    \abs{\det{M} - t} \leq \frac{n!\cdot (2n) \cdot \tau^{n}\cdot \delta'}{v}\cdot \abs{\det{M}}.
  \end{equation}
  \textbf{RUNTIME:} The algorithm performs at most \(\poly{n, \rep{M'}}\) bit
  operations.

\end{lemma}
\begin{proof}
  The algorithm simply outputs \(\det {M'}\), which can be computed using
  \(\poly{n, \rep{M'}}\) bit operations (via Gaussian elimination; see
  \cite{Edmondssystemsdistinctrepresentatives1967} and \cite[p.~39]{GLSbook}).
  To see that the guarantee on the output holds, we note that
  \begin{equation}
    \label{eq:19}
    \abs{\det{M} - \det{M'}} \leq n!\tau^n(\exp(n\delta') - 1) \leq n!\cdot(2n)\tau^n\delta',
  \end{equation}
  since \(n\delta' \leq \frac{1}{2}\).
\end{proof}

\begin{lemma}[\textbf{Approximating the logarithm}]
  \label{lem-log-approx}
  There is an algorithm with the following characteristics.\\
  \textbf{INPUT:} A rational number \(x > 0\) and a rational accuracy parameter
  \(\delta \in (0, 1)\).\\
  \textbf{OUTPUT:} A rational number \(t\) satisfying
  \begin{equation} \abs{\log x - t} \leq \delta.
  \end{equation}
  \textbf{RUNTIME:} The algorithm uses at most \(\poly{\rep{x, \delta}}\) bit
  operations.
\end{lemma}

\begin{proof}
  We begin by observing that given a positive rational number \(z \leq 1/3\) and a
  positive integer \(k\), the first \(k\) terms of the Taylor series around
  \(0\) of \(\log (1-z)\) can be computed in \(\poly{k, \rep{z}}\) bit
  operations, and provide a \(\pm\exp(-k)\) additive approximation of
  \(\log z\).

  We now assume without loss of generality that \(x > 1\), and construct, using
  at most \(\poly{\rep{x}}\) bit operations, a positive integer \(n\) and a
  rational number \(y\) such that \(y = x/(3/2)^n\) and
  \(\frac{2}{3} \leq y \leq 1\) (this can be done simply by trying out powers of
  \(\frac{3}{2}\) in increasing order).  Note that by construction, \(\rep{y}\)
  and \(n\) are both at most \(\poly{\rep{x}}\).  Thus, by the above
  observation, a \(\pm\delta/2\) additive approximation \(v_1\) to \(\log y\) can be
  computed in \(\poly{\log(1/\delta), \rep{x}}\) bit operations.  By the same
  observation, a \(\pm\delta/(2n)\) approximation \(v_2\) to \(\log (3/2)\) can be
  computed in \(\poly{\log(n/\delta)}\) bit operations.  This upper bound on
  bit operations, by the above upper bound on \(n\), is at most
  \(\poly{\log(1/\delta), \rep{x}}\). The quantity \(v_1 + n v_2\) then gives a
  \(\pm\delta\) approximation to \(\log x\), and the result follows since
  \(\log(1/\delta)\) is at most \(O(\rep{\delta})\).
\end{proof}

\begin{lemma}
  \label{lem-log-one-minux-exp-approx}
  Let \(z\) be a positive real number. Then
  \begin{equation}
    \abs{\log(1 - \exp(-z))} \leq \max\inb{1, \log(2/z)}.
  \end{equation}
\end{lemma}
\begin{proof}
  Note that
  \begin{equation}
    \label{eq:29}
    \abs{\log(1 - \exp(-z))} = \log\inp{1 + \frac{1}{\exp(z) - 1}} \leq
    \frac{1}{e^z - 1}.
  \end{equation}
  This is at most \(1\) when \(z \geq 1\).  On the other hand, when
  \(0 < z < 1\), we have \(1 > 1 - \exp(-z) \geq z/2\), and the claim follows.
\end{proof}

\section*{Acknowledgments}

We are grateful to Jonathan Leake and Igor Pak for helpful discussions. We are also grateful to Colin McSwiggen for many useful comments and a helpful discussion.

\notbool{anonymous}{\grantacknowledgement{}}

\end{document}